\documentclass[11pt]{article}
\usepackage{amssymb,amsmath}% font used for R in Real numbers
\usepackage{mathrsfs}
\usepackage{latexsym}
\usepackage{color}
\usepackage{amssymb}
\usepackage{amsmath}
\usepackage{amsthm}

\usepackage{amsmath,amsfonts,amsthm,amssymb,amscd,color,xcolor,mathrsfs,verbatim,microtype}

\usepackage{graphicx,eurosym}
\usepackage{hyperref}
\usepackage[applemac]{inputenc}
\usepackage[cyr]{aeguill}
\colorlet{darkblue}{blue!50!black}
\hypersetup{
    colorlinks,%
    citecolor=blue,%
    filecolor=red,%
    linkcolor=darkblue,%
    urlcolor=blue,%
    pdfnewwindow=true,%
    pdfstartview={FitH}
}

\allowdisplaybreaks

\catcode`!=11
\let\!int\int \def\int{\displaystyle\!int}
\let\!lim\lim \def\lim{\displaystyle\!lim}
\let\!sum\sum \def\sum{\displaystyle\!sum}
\let\!sup\sup \def\sup{\displaystyle\!sup}
\let\!inf\inf \def\inf{\displaystyle\!inf}
\let\!cap\cap \def\cap{\displaystyle\!cap}
\let\!max\max \def\max{\displaystyle\!max}
\let\!min\min \def\min{\displaystyle\!min}
\catcode`!=12

\newtheorem{theorem}{\bf Theorem}[section]
\newtheorem{lemma}{\bf Lemma}[section]
\newtheorem{definition}{\bf Definition}[section]
\newtheorem{proposition}{\bf Proposition}[section]

\newtheorem{remark}{\bf Remark}[section]

\catcode`@=11
\@addtoreset{equation}{section}
\catcode`@=12

\begin{document}

\title{Polynomial mixing for the white-forced hyperviscous Burgers equation on the whole line}
\author{Peng Gao
\\[2mm]
\small School of Mathematics and Statistics, and Center for Mathematics
\\
\small and Interdisciplinary Sciences, Northeast Normal University,
\\
\small Changchun 130024,  P. R. China
\\[2mm]
\small Email: gaopengjilindaxue@126.com }
\date{\today}
\maketitle

\vbox to -13truemm{}

\begin{abstract}
Our goal in this paper is to investigate ergodicity of the white-forced hyperviscous Burgers equation on the whole line.
Under the assumption that sufficiently many directions of the phase space are stochastically forced,
we can prove that the dynamics is attractive toward a unique invariant probability measure with polynomial rate of any
power. In order to prove this, we further develop coupling method and establish a sufficiently general criterion for polynomial mixing.
The proof of polynomial mixing for the white-forced hyperviscous Burgers equation is obtained by the combination of the coupling criterion and the Foia\c{s}-Prodi estimate of hyperviscous Burgers equation on the whole line.
\\[6pt]
\end{abstract}
\tableofcontents
\setcounter{section}{0}

\section{Introduction}
The \textit{hyperviscous Burgers equation} is given by
\begin{equation*}
u_{t}+u_{xxxx}+uu_{x}=0,
\end{equation*}
which introduces a fourth-order diffusion term $u_{xxxx}$, to model complex viscous effects in fluid flows. This equation arises as an extension of the classical Burgers equation
\begin{equation*}
u_{t}-u_{xx}+uu_{x}=0,
\end{equation*}
particularly suited for capturing small-scale turbulent phenomena that cannot be adequately described by standard viscosity.
In contrast, the classical Burgers equation includes only a second-order term $-u_{xx}$, representing traditional viscous dissipation. While it effectively describes wave propagation, shock formation, and basic viscous behavior, it often falls short in high Reynolds number regimes where turbulence involves more intricate diffusion processes. The hyperviscous Burgers equation addresses this limitation by incorporating higher-order viscosity, enabling a more accurate representation of microscopic turbulent effects, especially in regions with steep gradients or near boundaries. Thus, compared to its classical counterpart, the hyperviscous form provides a more powerful tool for analyzing and simulating nonlinear fluid dynamics where standard viscosity proves insufficient.
For further details on this aspect, we refer the reader to \cite{B2014-hyp,C1995-hyp,Z1997-hyp,F2013-hyp,G2002-hyp,B2019-hyp} and the references therein.
In \cite{C1995-hyp,G2002-hyp}, the hyperviscous Burgers equation driven by a random force with appropriately chosen spatial correlations was studied both numerically and via the dynamic renormalization group method. It was found that the statistical properties of its solutions closely resemble those observed experimentally in real three-dimensional turbulence. Since the mixing of fluid is a fundamental assumption in Kolmogorov's theory of turbulence, the theory of mixing has attracted widespread attention. Over the past three decades, substantial progress has been made in the study of mixing for partial differential equations (PDEs) driven by random forcing. We refer the reader to \cite{Fla1,Deb1,HM06,HM08,HM11-1,HM11-2,M1} and the monographs \cite{DaZ2,KS12} for comprehensive overviews of this field.

Inspired by \cite{B2014-hyp,C1995-hyp,Z1997-hyp,F2013-hyp,G2002-hyp,B2019-hyp}, one can use the white-forced hyperviscous Burgers equation as a model to study turbulence theory. As mentioned above in the context of Kolmogorov's theory of turbulence, the study of mixing for the white-forced hyperviscous Burgers equation on the entire real line $\mathbb{R}$ is of significant interest. Specifically, we elaborate on this from both physical and mathematical perspectives.
From a physical standpoint, within Kolmogorov's theory of turbulence, scholars often assume that fully developed turbulence in the whole space is spatially homogeneous and isotropic, and possesses a certain mixing property. This implies that, under the assumption of a sufficiently high Reynolds number and sufficiently strong turbulent mixing, one can estimate the same statistical quantity (such as the energy spectrum) either by taking a time average over a finite region or by taking a spatial average over a large domain. This assumption has long been widely applied in engineering and technology. Therefore, studying a turbulence model described by a stochastic partial differential equation (such as the stochastic hyperviscous Burgers equation) and establishing its mixing property on the whole space carries significant physical meaning, as it can provide a rigorous mathematical foundation for the above assumption.
From a mathematical perspective, previous research has mostly focused on the mixing properties of stochastic partial differential equations on bounded domains. The methods employed rely heavily on characteristics of bounded domains, such as compactness and the spectral properties of partial differential operators. However, these techniques are no longer applicable to the study of mixing properties on the whole space, which necessitates the development of new tools and techniques. Investigating the mixing properties of stochastic partial differential equations on the whole space can therefore help enrich the relevant mathematical research tools and methodologies.

More precisely, in the present paper, we consider the mixing problem for the following white-forced hyperviscous Burgers equation with damping
\begin{equation}\label{KSE}
\begin{array}{l}
\left\{
\begin{array}{lll}
u_{t}+au+u_{xxxx}+uu_{x}=h+\eta\\
u(x,0)=u_{0}
\end{array}
\right.
\end{array}
\begin{array}{lll}
\textrm{in}~\mathbb{R},
\\\textrm{in}~\mathbb{R},
\end{array}
\end{equation}
where $a>0$ is a constant, $h=h(x),$ $\eta$ is a white noise of the form
$$\eta(t):=\frac{\partial}{\partial t}W(t),~~W(t):=\sum\limits_{i=1}^{\infty}b_{i}\beta_{i}(t)e_{i}(x),$$
and $b_{i}\in \mathbb{R}, \{\beta_{i}\}_{i\geq1}$ is a sequence of
independent real-valued standard Brownian motions defined on a filtered probability space
$(\Omega,\mathcal{F},\mathcal{F}_{t},\mathbb{P})$ satisfying the usual conditions, $\{e_{i}\}_{i\geq 1}$ is an orthonormal basis in $L^{2}(\mathbb{R}).$

The reason this paper adopts the hyperviscous Burgers equation instead of the classical Burgers equation is primarily due to considerations of the dissipation mechanism on the whole line. The classical Burgers equation lacks sufficient dissipation on the whole line, making it difficult to effectively characterize the energy dissipation behavior of fluids in an unbounded domain. Specifically, on the whole line, the dissipative term $-u_{xx}$ in the classical Burgers equation provides weak dissipation and cannot effectively suppress the complexity introduced by the nonlinear term $uu_{x}$. This can be clearly observed in the construction of the Foia\c{s}-Prodi type estimate (see Theorem \ref{FP}). Due to this limitation in dissipation, current research on the classical Burgers equation on the whole space can only establish the existence and uniqueness of an invariant measure, but cannot determine its mixing rate.

To overcome the above difficulty, this paper introduces a higher-order viscous term, namely the hyperviscous Burgers equation. It should be noted that although \cite{N3} has studied mixing for the stochastic 2D Navier-Stokes system on the whole space, its methods and techniques cannot be directly applied to the Burgers equation on the whole line. The reason is that the approach for the 2D Navier-Stokes system relies heavily on the fact that the velocity field is divergence-free, which implies that there is no vorticity stretching. In contrast, the one-dimensional Burgers equation does feature vorticity stretching, rendering the methods for the 2D Navier-Stokes system inapplicable.

Therefore, to compensate for the insufficient dissipation of classical viscosity and the structural mismatch with the methods used for the 2D Navier-Stokes system, we adopt a strategy of enhancing dissipation-namely, taking the hyperviscous Burgers equation, rather than the classical Burgers equation, as the object of study.

Let $H:=L^{2}(\mathbb{R}), V:=H^{1}(\mathbb{R})$. We introduce the following assumption
\begin{equation*}
\begin{array}{l}
\textbf{(A)}
\left\{
\begin{array}{lll}
&\varphi h\in H,~~\sum\limits_{i=1}^{\infty}|(h,e_{i})|\|e_{i}\|_{3}<+\infty,~~B_{1}:=\sum\limits_{i=1}^{\infty}b_{i}^{2}<+\infty,\\
&B_{2}:=\sum\limits_{i=1}^{\infty}b_{i}^{2}\|\varphi e_{i}\|^{2}<+\infty,~~B_{3}:=\sum\limits_{i=1}^{\infty}b_{i}^{2}\|e_{i}\|_{3}^{2}<+\infty,
\end{array}
\right.
\end{array}
\end{equation*}
where $\varphi(x):=\ln (x^{2}+2)$. We refer the reader to the notation at the end of Section 1 for a detailed definition of $\|\cdot\|_s$.
Under the assumption $\textbf{(A)}$, following the similar arguments as in \cite{DaZ1,F1}, we can show that the stochastic hyperviscous Burgers equation \eqref{KSE} is globally well-posed and defines a Markov process in $V$.
\par
~~
\par
Now, we are in a position to present the main result in this paper.
\begin{theorem}\label{MT1}
Let the assumption $\textbf{(A)}$ hold. Then, there exists an integer $N\geq1$ such that if
\begin{equation}\label{41}
\begin{split}
b_{i}\neq0,~i=1,2,\cdots,N,
\end{split}
\end{equation}
then there exists a unique stationary measure $\mu\in \mathcal{P}(V)$ for \eqref{KSE}. Moreover, for any $p>1$, there exists a $C_{p}>0$ such that for any $u_{0}\in V,$ the solution $u$ of \eqref{KSE} satisfies
\begin{equation*}
\begin{split}
\|P_{t}(u_{0},\cdot)-\mu\|_{L(V)}^{*}\leq C_{p}(1+\|u_{0}\|_{1}^{2}+\|u_{0}\|^{\frac{14}{3}})(t+1)^{-p},~t\geq 0.
\end{split}
\end{equation*}
\end{theorem}
We refer the reader to the notation at the end of Section 1 for a detailed definition of $\|\cdot\|_{L(V)}^{*}$.
In order to prove Theorem \ref{MT1}, we develop an abstract criterion for establishing polynomial mixing of Markov processes.
\par
Let $X$ be a separable Banach space with a norm $\|\cdot\|$. Let $(u_{t}, \mathbb{P}_{u})$ be a Feller family of Markov processes in $X,$ $P_{t}(u,A):=\mathbb{P}(u_{t}(u,\omega)\in A)$ is
the transition function. We introduce the following associated Markov operators
\begin{equation*}
\begin{split}
&\mathcal{B}_{t}: C_{b}(X)\rightarrow C_{b}(X),~\mathcal{B}_{t}f(u):=\int_{X}f(v)P_{t}(u,dv),~\forall f\in C_{b}(X),\\
&\mathcal{B}_{t}^{*}: \mathcal{P}(X)\rightarrow \mathcal{P}(X),~\mathcal{B}_{t}^{*}\lambda(A):=\int_{X}P_{t}(u,A)\lambda(du),~\forall \lambda\in \mathcal{P}(X).
\end{split}
\end{equation*}

\begin{definition}
Let $(\mathbf{u}_{t},\mathbb{P}_{\mathbf{u}})$ be a Markov family in $X\times X.$
$(\mathbf{u}_{t},\mathbb{P}_{\mathbf{u}})$ is called an \textbf{extension} of $(u_{t}, \mathbb{P}_{u})$ if for any $\mathbf{u}=(u,u^{\prime})\in X\times X$ the laws under $\mathbb{P}_{\mathbf{u}}$ of processes $\{\Pi_{1}\mathbf{u}_{t}\}_{t\geq0}$ and $\{\Pi_{2}\mathbf{u}_{t}\}_{t\geq0}$ coincide with
those of $\{u_{t}\}_{t\geq0}$ under $\mathbb{P}_{u}$ and $\mathbb{P}_{u^{\prime}}$ respectively, where $\Pi_{1}$ and $\Pi_{2}$ denote the
projections from $X\times X$ to the first and second component.
\end{definition}
Let $B$ be a closed subset in $X$, $\textbf{X}:=X\times X$ and $\textbf{B}:=B\times B.$ Define
\begin{equation*}
\begin{split}
&\tau_{\textbf{B}}=\tau_{\textbf{B}}(\textbf{u},\omega)=\inf\{t\geq0:\textbf{u}_{t}(\textbf{u},\omega)\in \textbf{B}\},\\
&\sigma=\sigma(\textbf{u},\omega)=\inf\{t\geq0:\|u_{t}(\textbf{u},\omega)-u_{t}^{\prime}(\textbf{u},\omega)\|\geq C(t+1)^{-p}\},\\
&\rho=\sigma+\tau_{\textbf{B}}\circ \theta_{\sigma}=\sigma(\textbf{u},\omega)+\tau_{\textbf{B}}(\textbf{u}_{\sigma(\textbf{u},\omega)}(\textbf{u},\omega),\theta_{\sigma(\textbf{u},\omega)}\omega),
\end{split}
\end{equation*}
where $\{\theta_t\}_{t\geq 0}$ is a semigroup of measure-preserving transformations from $\Omega$ to $\Omega$ such that $\theta_t^{-1}\mathcal{F}_s\subset \mathcal{F}_{t+s}$, see \cite{Shi08}.

\begin{definition}\label{Def1}
The family $(u_{t},\mathbb{P}_{u})$ is called satisfies the \textbf{coupling
hypothesis} if there exists an extension $(\mathbf{u}_{t},\mathbb{P}_{\mathbf{u}})$, a closed set $B\subset X,$ and an increasing
function $g(r)\geq1$ of the variable $r\geq0$ such that the following two properties
hold
\par
\textbf{Recurrence} There exists $p>1$ such that
$$\mathbb{E}_{\mathbf{u}} \tau_{\textbf{B}}^{p}\leq G(\mathbf{u})$$ for all $\mathbf{u}=(u,u^{\prime})\in \mathbf{X},$
where we set $G(\mathbf{u})=g(\|u\|)+g(\|u^{\prime}\|).$
\par
\textbf{Polynomial squeezing} There exist positive constants $\delta_{1},\delta_{2},c,K$ and $1\leq q\leq p$
such that, for any $\mathbf{u}\in \mathbf{B},$ we have
\begin{eqnarray}
&\label{PS1}\mathbb{P}_{\mathbf{u}}(\sigma=\infty)\geq \delta_{1},\\
&\label{PS2}\mathbb{E}_{\mathbf{u}}(\mathbb{I}_{\sigma<\infty}\sigma^{p})\leq c,\\
&\label{PS3}\mathbb{E}_{\mathbf{u}}(\mathbb{I}_{\sigma<\infty}G(\mathbf{u}_{\sigma})^{q})\leq K.
\end{eqnarray}
Any extension of $(u_{t},\mathbb{P}_{u})$ satisfying the above properties will be called a \textbf{mixing extension}.
\end{definition}
\par
Now, we are in a position to present an abstract criterion for polynomial mixing. Let $\mathcal{T}=\mathbb{R}$ or $\mathbb{Z}$ and
$\mathcal{T}_+=\{t\in \mathcal{T}:~t\geq 0\}$.

\begin{theorem}\label{THACPM}
Let $(u_{t},\mathbb{P}_{u})$ be a Feller family of Markov processes that possesses
a mixing extension $(\mathbf{u}_{t},\mathbb{P}_{\mathbf{u}})$. Then there exists a random time $\ell\in \mathcal{T}_{+}$ such
that, for any $\mathbf{u}\in \mathbf{X},$ with $\mathbb{P}_{\mathbf{u}}$-probability $1$, for any $p_{0}<p,$ we have
\begin{eqnarray}
&\label{PS4}\|u_{t}-u_{t}^{\prime}\|\leq C_{1}(t-\ell+1)^{-p},~\forall t\geq \ell,\\
&\nonumber\mathbb{E}_{\mathbf{u}}\ell^{p_{0}}\leq C_{1}(G(\mathbf{u})+1),
\end{eqnarray}
where $\mathbf{u}\in \mathbf{X}$ is an arbitrary initial point, $g(r)$ is the function in Definition \ref{Def1},
and $C_{1}, \alpha,$ and $\beta$ are positive constants not depending on $\mathbf{u}$ and $t.$ If, in addition,
there is an increasing function $\tilde{g}(r)\geq1$ such that
\begin{equation*}
\begin{split}
\mathbb{E}_{u}g(\|u_{t}\|)\leq \tilde{g}(u),~\forall t\geq 0,
\end{split}
\end{equation*}
for all $u\in X,$ then the family $(u_{t},\mathbb{P}_{u})$ has a unique stationary measure $\mu\in \mathcal{P}(X),$ and it holds that
\begin{equation*}
\begin{split}
\|P_{t}(u,\cdot)-\mu\|^{*}_{L}\leq M(\|u\|)(t+1)^{-p_{0}},~~\forall t\geq 0,
\end{split}
\end{equation*}
for $u\in X,$
where $M$ is given by the relation
$$M(r)=3C_{1}(g(r)+\tilde{g}(0)).$$
\end{theorem}
\begin{remark}
Analysing the proof of Theorem \ref{THACPM}, it is not difficult to see that
Theorem \ref{THACPM} remains valid if $\sigma$ is replaced by any other stopping time $\tilde{\sigma}$ such that
$$\mathbb{P}_{\textbf{u}}(\tilde{\sigma}\leq \sigma)=1~for~any~\textbf{u}\in \textbf{B}.$$
In other words, if inequalities \eqref{PS1}-\eqref{PS4} hold with $\sigma$ replaced by $\tilde{\sigma}$, then
the conclusion of Theorem \ref{THACPM} is true. To see this, it suffices to repeat the
arguments in the proof of Theorem \ref{THACPM}, replacing $\sigma$ by $\tilde{\sigma}$ everywhere.
\end{remark}
\begin{remark}
Theorem \ref{THACPM} may be viewed as a polynomial counterpart of Theorem 3.1.7 in \cite{KS12}. Whereas Theorem 3.1.7 establishes exponential mixing for Markov processes under specific conditions, Theorem \ref{THACPM} generalizes these methods and concepts to polynomial mixing, thereby providing a framework to investigate the mixing behavior of a wider class of stochastic partial differential equations.
\end{remark}

Early work primarily focused on establishing mixing for SPDEs on the bounded domain, including the Navier-Stokes system, the Burgers equation, the complex Ginzburg-Landau equation, and the reaction-diffusion equation. Indeed, much of the existing literature is devoted to exponential mixing for SPDEs; see, for instance, \cite{HM08,HM11-1,FG15,KS12,KNS1,Shi08} and the references therein. The exponential rates obtained in these works rely crucially on the strong dissipative structure of the underlying systems. In contrast, establishing comparable results for weakly dissipative SPDEs (e.g., the white-forced parabolic equation on the whole space, the Schr\"{o}dinger equation, and the wave equation, among others) is generally far more challenging.
In recent years, increasing attention has been devoted to polynomial mixing for weakly dissipative SPDEs, where the absence of strong dissipation renders exponential convergence unlikely. To date, however, polynomial mixing has been investigated to a much more limited extent (see, e.g., \cite{PMJEE2005,PMJEE2008,PMAMO,PM2024}). All of these works are concerned with equations posed on bounded spatial domains. We emphasize that, compared with exponential mixing, the analysis of polynomial mixing is substantially more delicate, reflecting the weaker dissipative mechanisms inherent in the systems under consideration.

In recent years, the mixing problem for SPDEs on the whole space has gradually attracted increasing attention (see, e.g., \cite{BCK14,BL19,DGR21,BFZ23,N2,N3,GPW}).
To the best of our knowledge, in previous works on the mixing property of stochastic Burgers equations on the whole line, only the existence and uniqueness of invariant measures were typically established, without corresponding results on convergence rates. In the present paper, we develop the coupling method to the hyperviscous Burgers equation. As a result, we not only establish the uniqueness of the invariant measure for the hyperviscous Burgers equation but also obtain a polynomial convergence rate (see Theorem~\ref{MT1}), thereby laying a foundation for further research on related problems. Theorem \ref{THACPM}, established in this paper, is also of independent interest, as it provides an abstract criterion for establishing polynomial mixing of Markov processes, this abstract criterion has subsequently been applied to the study of polynomial mixing for stochastic wave equations (see \cite{GPW}) on the whole line and stochastic Navier-Stokes systems (see \cite{N3}) on the whole space.

In the present paper, the principal difficulty in establishing mixing for the white-forced hyperviscous Burgers equation \eqref{KSE} on the whole line stems from the nonlinear term $uu_{x}$ and the unboundedness of the spatial domain. These difficulties lead to polynomial, rather than exponential, mixing of the system.
To prove Theorem~\ref{MT1}, we further develop the coupling method and rely on Theorem~\ref{THACPM}. More precisely, when applying Theorem~\ref{THACPM} to the hyperviscous Burgers equation, we combine the coupling approach with a Foia\c{s}-Prodi type estimate on the whole line. This strategy relies crucially on a suitable weight estimate, which is substantially more delicate to derive for the hyperviscous Burgers equation than for equations such as the Schr\"{o}dinger equation, the complex Ginzburg-Landau equation, and the reaction-diffusion equation.
In contrast to the nonlinear terms $u^{3}$ or $|u|^{2}u$ found in these equations, where one can typically rely on the sign condition, the term $uu_{x}$ cannot be treated in the same way.
Instead, we handle this term by employing high-order moment estimates (see Section 2.2 and Section 2.3) together with a maximal martingale inequality, which is fundamentally different from traditional ones. Moreover, in the absence of strong dissipation and compactness, techniques commonly used for SPDEs on bounded domains, such as the Poincar\'{e} inequality, are no longer directly applicable in the unbounded domain setting. To overcome this difficulty, we introduce a cut-off technique tailored to the hyperviscous Burgers equation \eqref{KSE}. Leveraging these new ideas and methods, we are able to establish Theorem~\ref{MT1}.

\par~~
\par
\textit{Notation}
\par
Throughout the rest of the paper, $c$ and $C$ denote generic positive constants that may change
from line to line. The main parameters that they depend on will appear between parenthesis, e.g.,
$C(p,T)$ is a function of $p$ and $T.$
\par
Let $X$ be a Polish space with a metric $d_{X}(u,v)$, the Borel $\sigma$-algebra on $X$ is denoted by $\mathcal{B}(X)$ and the set of Borel
probability measures by $\mathcal{P}(X).$ $C_{b}(X)$ is the space of continuous functions $f:X\rightarrow \mathbb{C}$ endowed with the
norm $\|f\|_{\infty}=\sup_{u\in X}|f(u)|.$ $B_{X}(R)$ stands for the ball in $X$ of radius $R$ centred at zero.
We write $C(X)$ when $X$ is compact. $L_{b}(X)$ is the space of functions $f\in C_{b}(X)$ such that
$$\|f\|_{L(X)}=\|f\|_{\infty}+\sup\limits_{u\neq v}\frac{|f(u)-f(v)|}{d_{X}(u,v)}<\infty.$$
The dual-Lipschitz metric on $\mathcal{P}(X)$ is defined by
$$\|\mu_{1}-\mu_{2}\|^{*}_{L(X)}=\sup_{\|f\|_{L(X)}\leq 1}|\langle f,\mu_{1}\rangle-\langle f,\mu_{2}\rangle|,~~\mu_{1},\mu_{2}\in \mathcal{P}(X),$$
where $\langle f,\mu \rangle=\int_{X}f(u)\mu(du).$
\par
Let $H^{s}:=H^{s}(\mathbb{R})$ we denote the Sobolev space of real-valued functions on $\mathbb{R}$ with the norm
\begin{equation*}
\begin{split}
\|u\|_{s}^{2}:=\int_{\mathbb{R}}(1+|\xi|^{2})^{s}|(\mathcal{F}u)(\xi)|^{2}d\xi,
\end{split}
\end{equation*}
where $\mathcal{F}u$ is the Fourier transform of $u.$ When $s=0$, we denote $\|\cdot\|_{0}$ by $\|\cdot\|$.
\par
For $q>1,$ $Q_{q}(r):=r^{1-q}(r>0)$.
\par
If $p$ and $q$ are real numbers, then $p\vee q (p\wedge q)$ stands for their maximum(minimum).
\par
$a_{0}:=a\wedge 1,$ where $a$ is the damping coefficient in \eqref{KSE}.

\section{Probability estimates}
\subsection{Construct a mixing extension}
For initial points $u,u^{\prime}\in H,$ let $\{u(t)\}_{t\geq 0}$ and $\{u^{\prime}(t)\}_{t\geq 0}$
denote the solutions of the equation \eqref{KSE} issued from $u,u^{\prime},$ respectively. We introduce
an auxiliary process $\{v(t)\}_{t\geq 0}$ which is the solution of the problem
\begin{eqnarray}\label{F1}
 \begin{array}{l}
 \left\{
 \begin{array}{llll}
 v_{t}+av+v_{xxxx}+vv_{x}+P_{N}(uu_{x}-vv_{x})=h+\eta
 \\v(x,0)=u^{\prime}(x)
 \end{array}
 \right.
 \end{array}
 \begin{array}{lll}
 {\rm{in}}~\mathbb{R},
\\{\rm{in}}~\mathbb{R},
\end{array}
\end{eqnarray}
where $P_{N}$ denotes the orthogonal projection in $H$ onto the space by the family $\{e_{1},\cdots,e_{N}\}$
with integer $N\geq 1$. Let $T>0$ be a time parameter that will be chosen later. We denote by
$\lambda_{T}(u,u^{\prime})$ and $\lambda_{T}^{\prime}(u,u^{\prime})$ the distributions of processes $\{v(t)\}_{0\leq t\leq T}$ and $\{u^{\prime}(t)\}_{0\leq t\leq T}$,
respectively. Then $\lambda_{T}(u,u^{\prime})$ and $\lambda_{T}^{\prime}(u,u^{\prime})$ are probability measures on $C([0,T];H).$
By \cite[Theorem 1.2.28]{KS12}, there exists a maximal coupling $(V_{T}(u,u^{\prime}),V_{T}^{\prime}(u,u^{\prime}))$
for the pair $(\lambda_{T}(u,u^{\prime}),\lambda_{T}^{\prime}(u,u^{\prime}))$ defined on some probability space $(\tilde{\Omega},\tilde{\mathcal{F}},\tilde{\mathbb{P}})$.
We denote by $\{\tilde{v}(t)\}_{0\leq t\leq T}$ and $\{\tilde{u}^{\prime}(t)\}_{0\leq t\leq T}$ the flows of this maximal coupling.
Then, $\tilde{v}$ is the solution of
\begin{eqnarray}\label{29}
 \begin{array}{l}
 \left\{
 \begin{array}{llll}
 \tilde{v}_{t}+a\tilde{v}+\tilde{v}_{xxxx}+\tilde{v}\tilde{v}_{x}-P_{N}(\tilde{v} \tilde{v}_{x})=h+\tilde{\eta}
 \\\tilde{v}(x,0)=u^{\prime}(x)
 \end{array}
 \right.
 \end{array}
 \begin{array}{lll} {\rm{in}}~\mathbb{R},
\\{\rm{in}}~\mathbb{R},
\end{array}
\end{eqnarray}
with $\mathcal{D}(\{\int_{0}^{t}\tilde{\eta}(s)ds\}_{0\leq t\leq T})=\mathcal{D}(\{\int_{0}^{t}(\eta(s)-P_{N}(u(s)u_{x}(s))ds\}_{0\leq t\leq T})$.
Let $\tilde{u}$ be the solution of
\begin{eqnarray}\label{30}
 \begin{array}{l}
 \left\{
 \begin{array}{llll}
 \tilde{u}_{t}+a\tilde{u}+\tilde{u}_{xxxx}+\tilde{u}\tilde{u}_{x}-P_{N}(\tilde{u}\tilde{u}_{x})=h+\tilde{\eta}
 \\\tilde{u}(x,0)=u(x)
 \end{array}
 \right.
 \end{array}
 \begin{array}{lll} {\rm{in}}~\mathbb{R},
\\{\rm{in}}~\mathbb{R},
\end{array}
\end{eqnarray}
this implies that $\mathcal{D}(\{\tilde{u}(t)\}_{0\leq t\leq T})=\mathcal{D}(\{u(t)\}_{0\leq t\leq T})$.
We define operators $\mathcal{R}$ and $\mathcal{R}^{\prime}$ by
\begin{equation*}
\begin{split}
\mathcal{R}_{t}(u,u^{\prime},\omega)=\tilde{u}(t),~~\mathcal{R}_{t}^{\prime}(u,u^{\prime},\omega)=\tilde{u}^{\prime}(t)
\end{split}
\end{equation*}
for any $u,u^{\prime}\in H, \omega\in \tilde{\Omega}, t\in [0,T]$. Then, let
$\{(\Omega^{k},\mathcal{F}^{k},\mathbb{P}^{k})\}_{k\geq0}$ be
a sequence of independent copies of $(\tilde{\Omega},\tilde{\mathcal{F}},\tilde{\mathbb{P}})$ and let $(\Omega,\mathcal{F},\mathbb{P})$ be the direct
product of $\{(\Omega^{k},\mathcal{F}^{k},\mathbb{P}^{k})\}_{k\geq1}.$ For any $\omega=(\omega^{1},\omega^{2},\cdots)\in \Omega$
and $u,u^{\prime}\in H,$ we set $\tilde{u}_{0}=u,\tilde{u}_{0}^{\prime}=u^{\prime},$ and
\begin{equation*}
\begin{split}
\tilde{u}_{t}(\omega)=\mathcal{R}_{s}(\tilde{u}_{k}(\omega),\tilde{u}_{k}^{\prime}(\omega),\omega^{k}),&
~~\tilde{u}^{\prime}_{t}(\omega)=\mathcal{R}_{s}^{\prime}(\tilde{u}_{k}(\omega),\tilde{u}_{k}^{\prime}(\omega),\omega^{k}),\\
&\mathbf{u}_{t}=(\tilde{u}_{t},\tilde{u}^{\prime}_{t}),
\end{split}
\end{equation*}
where $t=s+kT, s\in[0,T).$ The construction implies that $(\mathbf{u}_{t}, \mathbb{P}_{\mathbf{u}})$ is an
extension for $(u_{t}, \mathbb{P}_{u}).$ We will show that it is a mixing extension for $(u_{t}, \mathbb{P}_{u}).$

\subsection{Moment estimates for solutions}
\begin{proposition}\label{pro5}
For any $p\geq1$, $u_{0}\in H, h\in H$, it holds that
\begin{equation}\label{2}
\begin{split}
\mathbb{E}\|u(t)\|^{2p}\leq e^{-pat}\|u_{0}\|^{2p}+\frac{C(p,h,B_{1})}{pa}~~\forall t\geq0
\end{split}
\end{equation}
and
\begin{equation}\label{5}
\begin{split}
\mathbb{E}\int_{0}^{t}\|u(s)\|^{2p}ds\leq C(u_{0},p,h,B_{1})(t+1)~~\forall t\geq0.
\end{split}
\end{equation}
\end{proposition}
\begin{proof}[Proof of Proposition \ref{pro5}]
By applying It\^{o} formula to $\|u\|^{2}$, it holds that
\begin{equation}\label{1}
\begin{split}
d\|u\|^{2}+2[a\|u\|^{2}+\|u_{xx}\|^{2}]dt=2(h,u)dt+B_{1}dt+2(u,dW),
\end{split}
\end{equation}
where $B_{1}=\sum\limits_{i=1}^{\infty}b_{i}^{2}$. This leads to
\begin{equation*}
\begin{split}
\frac{d}{dt}\mathbb{E}\|u\|^{2}+2[a\mathbb{E}\|u\|^{2}+\mathbb{E}\|u_{xx}\|^{2}]=2\mathbb{E}(h,u)+B_{1}\leq a\mathbb{E}\|u\|^{2}+\frac{\|h\|^{2}}{a}+B_{1},
\end{split}
\end{equation*}
this implies that
\begin{equation*}
\begin{split}
\frac{d}{dt}\mathbb{E}\|u\|^{2}+a\mathbb{E}\|u\|^{2}+2\mathbb{E}\|u_{xx}\|^{2}\leq \frac{\|h\|^{2}}{a}+B_{1},
\end{split}
\end{equation*}
thus, Gronwall inequality implies that
\begin{equation*}
\begin{split}
\mathbb{E}\|u(t)\|^{2}\leq e^{-at}[\|u_{0}\|^{2}+\frac{1}{a}(\frac{\|h\|^{2}}{a}+B_{1})].
\end{split}
\end{equation*}
\par
According to \eqref{1} and applying It\^{o} formula to $\|u\|^{2p}$, we can obtain that
\begin{equation*}
\begin{split}
d\|u\|^{2p}=&[2p\|u\|^{2p-2}(u,h-au-u_{xxxx}-uu_{x})+p\|u\|^{2p-2}B_{1}]dt
\\&+2p(p-1)\|u\|^{2(p-2)}(u,dW)^{2}+2p\|u\|^{2p-2}(u,dW).
\end{split}
\end{equation*}
Since $(u,dW)^{2}=\|u\|^{2}B_{1}dt$, we have
\begin{equation*}
\begin{split}
d\|u\|^{2p}=[2p\|u\|^{2p-2}(u,h-au-u_{xxxx}-uu_{x})+(2p^{2}-p)\|u\|^{2p-2}B_{1}]dt
+2p\|u\|^{2p-2}(u,dW).
\end{split}
\end{equation*}
This leads to
\begin{equation}\label{4}
\begin{split}
\|u(t)\|^{2p}+pa\int_{s}^{t}\|u(r)\|^{2p}dr\leq\|u(s)\|^{2p}+C(p,\|h\|,B_{1})(t-s)+2p\int_{s}^{t}\|u\|^{2p-2}(u,dW)
\end{split}
\end{equation}
and
\begin{equation*}
\begin{split}
\frac{d}{dt}\mathbb{E}\|u\|^{2p}&=2p\mathbb{E}\|u\|^{2p-2}(u,h-au-u_{xxxx}-uu_{x})+(2p^{2}-p)B_{1}\mathbb{E}\|u\|^{2p-2}\\
&=2p\mathbb{E}\|u\|^{2p-2}[(u,h)-a\|u\|^{2}-\|u_{xx}\|^{2}]+(2p^{2}-p)B_{1}\mathbb{E}\|u\|^{2p-2},
\end{split}
\end{equation*}
then,
\begin{equation*}
\begin{split}
\frac{d}{dt}\mathbb{E}\|u\|^{2p}+2pa\mathbb{E}\|u\|^{2p}+2p\mathbb{E}\|u\|^{2p-2}\|u_{xx}\|^{2}
&=2p\mathbb{E}\|u\|^{2p-2}(u,h)+(2p^{2}-p)B_{1}\mathbb{E}\|u\|^{2p-2}\\
&\leq \frac{pa}{4}\mathbb{E}\|u\|^{2p}+C(p,h,B_{1}),
\end{split}
\end{equation*}
this leads to
\begin{equation*}
\begin{split}
\frac{d}{dt}\mathbb{E}\|u\|^{2p}+pa\mathbb{E}\|u\|^{2p}+2p\mathbb{E}\|u\|^{2p-2}\|u_{xx}\|^{2}
\leq C(p,h,B_{1}),
\end{split}
\end{equation*}
the above estimates imply that \eqref{2} and \eqref{5} hold.
\par
This completes the proof.
\end{proof}

\subsection{Weighted estimates for solutions}
Inspired by \cite{A1989,N2}, let us introduce a smooth space-time weight function given by
\begin{equation*}
\begin{split}
\psi(t,x):=\varphi(x)(1-\exp(-\frac{t}{\varphi(x)})), ~~(t,x)\in \mathbb{R}^{2},
\end{split}
\end{equation*}
where $\varphi(x):=\ln (x^{2}+2)$. The following properties are useful for establishing Foia\c{s}-Prodi estimate.
\par
(1) $0<\psi(t,x)<\varphi(x)$ for $t>0$ and $\psi(0,x)=0$;
\par
(2) For any $k\geq1,$ $|\partial_{k}\psi|\leq C_{k}$, where $C_{k}$ is some constants;
\par
(3) $\lim\limits_{t,|x|\rightarrow +\infty}\psi(t,x)=+\infty.$
\par
We introduce the following functionals
\begin{equation*}
\begin{split}
&\mathcal{E}_{u}(t)=\|u(t)\|^{2}+a_{0}\int_{0}^{t}\|u\|_{2}^{2}ds,\\
&\mathcal{E}^{p}_{u}(t)=\|u(t)\|^{2p}+pa\int_{0}^{t}\|u\|^{2p}ds,\\
&\mathcal{F}^{\psi}_{u}(t)=\|(\psi u)(t)\|^{2}+a_{0}\int_{0}^{t}(\|\psi u_{xx}\|^{2}+\|\psi u\|^{2})ds,\\
&\mathcal{G}^{\psi}_{u}(t)=\|u(t)\|^{2}+\|(\psi u)(t)\|^{2}+a_{0}\int_{0}^{t}(\|u\|_{2}^{2}+\|\psi u\|_{2}^{2})ds,\\
&\mathcal{E}^{\psi}_{u}(t)=\mathcal{E}^{p_{1}}_{u}(t)+\mathcal{G}^{\psi}_{u}(t),~~{\rm{where}}~p_{1}=\frac{7}{3}.
\end{split}
\end{equation*}

\begin{proposition}\label{pro1}
There exist constants $K_{1},\gamma_{1}>0$ such that
\begin{equation}
\begin{split}
\mathbb{P}(\sup\limits_{t\geq 0}[\mathcal{E}_{u}(t+T)-\mathcal{E}_{u}(T)-K_{1}t]\geq \rho)\leq e^{-\gamma_{1}\rho},
\end{split}
\end{equation}
for all $\rho>0,T\geq 0,$ and $u\in H.$
\end{proposition}

\begin{proof}[Proof of Proposition \ref{pro1}]
It follows from \eqref{1} that
\begin{equation*}
\begin{split}
\mathcal{E}_{u}(t+T)-\mathcal{E}_{u}(T)+2\int_{T}^{t+T}[a\|u(s)\|^{2}+\|u_{xx}(s)\|^{2}]ds
= 2\int_{T}^{t+T}(u,h)ds+B_{1}t+M(t),
\end{split}
\end{equation*}
where $M(t):=2\int_{T}^{t+T}(u,dW)=2\int_{T}^{t+T}\sum_{i=1}^{\infty}b_{i}(u,e_{i})d\beta_{i},$ it is easy to see that $M(t)$ is
a martingale and its quadratic variation is given by
\begin{equation*}
\begin{split}
\langle M\rangle(t)=4\int_{T}^{t+T}\sum_{i=1}^{\infty}b_{i}^{2}(u(s),e_{i})^{2}ds\leq 4B_{1}\int_{T}^{t+T}\|u(s)\|^{2}ds.
\end{split}
\end{equation*}
Noting the fact
\begin{equation*}
\begin{split}
\int_{T}^{t+T}(u,h)ds\leq a\int_{T}^{t+T}\|u(s)\|^{2}ds+\frac{\|h\|^{2}}{a}t,
\end{split}
\end{equation*}
there exist some suitable positive constants $K_{1}$ and $\gamma_{1}$ such that
\begin{equation*}
\begin{split}
\mathcal{E}_{u}(t+T)-\mathcal{E}_{u}(T)\leq K_{1}t+M(t)-\frac{\gamma_{1}}{2}\langle M\rangle(t),
\end{split}
\end{equation*}
exponential supermartingale inequality implies that
\begin{equation*}
\begin{split}
\mathbb{P}(\sup_{t\geq0}(\mathcal{E}_{u}(t+T)-K_{1}t)\geq \mathcal{E}_{u}(T)+\rho)
&\leq \mathbb{P}(\sup_{t\geq0}(\gamma_{1} M(t)-\frac{\gamma_{1}^{2}}{2}\langle M\rangle(t))\geq \gamma_{1}\rho)\\
&= \mathbb{P}(\sup_{t\geq0}\exp(\gamma_{1} M(t)-\frac{1}{2}\langle \gamma_{1} M\rangle(t))\geq e^{\gamma_{1}\rho})\\
&\leq e^{-\gamma_{1} \rho}.
\end{split}
\end{equation*}
\par
This completes the proof.
\end{proof}

\begin{proposition}\label{pro2}
For any $p\geq2, q>1$, there exists a constant $L_{p}>0$ such that
\begin{equation*}
\begin{split}
\mathbb{P}(\sup\limits_{t\geq 0}[\mathcal{E}^{p}_{u}(t+T)-\mathcal{E}^{p}_{u}(T)-(L_{p}+1)t-2]\geq \rho)\leq CQ_{q}(\rho+1),
\end{split}
\end{equation*}
for all $\rho>0,T\geq 0,$ and $u\in H$, where $C$ is a constant depends on $u,q,h,B_{1},B_{2}$.
\end{proposition}
\begin{proof}[Proof of Proposition \ref{pro2}]
It follows from \eqref{4} that
\begin{equation}\label{7}
\begin{split}
\mathcal{E}^{p}_{u}(t+T)-\mathcal{E}^{p}_{u}(T)&=\|u(t+T)\|^{2p}-\|u(T)\|^{2p}+pa\int_{T}^{t+T}\|u(s)\|^{2p}ds\\
&\leq L_{p}t+M^{p}(t),
\end{split}
\end{equation}
where $M^{p}(t):=2p\int_{T}^{t+T}\|u\|^{2p-2}(u,dW)$. It follows from \eqref{7} that
\begin{equation*}
\begin{split}
\mathcal{E}^{p}_{u}(t+T)-\mathcal{E}^{p}_{u}(T)-(L_{p}+1)t-2\leq M^{p}(t)-t-2.
\end{split}
\end{equation*}
\par
We define $M^{p}_{*}(t)=\sup_{s\in [0,t]}|M^{p}(s)|.$
For any $q>1$, according to the Burkholder-Davis-Gundy and H\"{o}lder inequalities, we have
\begin{equation}\label{6}
\begin{split}
\mathbb{E}(M^{p}_{*}(t))^{2q}
&\leq C\mathbb{E}\langle M^{p}(t)\rangle^{q}
\leq C(p,q)\mathbb{E}(\int_{T}^{t+T}\|u\|^{4p-2}ds)^{q}\\
&\leq C(p,q)t^{q-1}\mathbb{E}\int_{T}^{t+T}\|u\|^{(4p-2)q}ds
\leq C(u_{0},p,q,h,B_{1})(t+1)^{q}.
\end{split}
\end{equation}
Noting the following fact for all $\rho>0,$ it holds that
\begin{equation*}
\begin{split}
\{\sup\limits_{t\geq 0}[M^{p}(t)-t-2]\geq \rho\}
\subset &\bigcup\limits_{m\geq 0}\{\sup\limits_{t\in [m,m+1)}[M^{p}(t)-t-2]\geq \rho\}\\
\subset &\bigcup\limits_{m\geq 0}\{M_{*}^{p}(m+1)\geq \rho+m+2\}.
\end{split}
\end{equation*}
For any $q>1$, it follows from the above analysis, the maximal martingale inequality, the Chebyshev inequality and \eqref{6} that
\begin{equation*}
\begin{split}
\mathbb{P}\{\sup\limits_{t\geq 0}[M^{p}(t)-t-2]\geq \rho\}
&\leq\mathbb{P}(\bigcup\limits_{m\geq 0}\{M_{*}^{p}(m+1)\geq \rho+m+2\})\\
&\leq\sum\limits_{m\geq 0}\mathbb{P}(M_{*}^{p}(m+1)\geq \rho+m+2)\\
&\leq\sum\limits_{m\geq 0}\frac{\mathbb{E}(M_{*}^{p}(m+1))^{2q}}{(\rho+m+2)^{2q}}\\
&\leq C(u_{0},p,q,h,B_{1})\sum\limits_{m\geq 0}\frac{(m+2)^{q}}{(\rho+m+2)^{2q}}\\
&\leq C(u_{0},p,q,h,B_{1})\sum\limits_{m\geq 0}\frac{1}{(\rho+m+2)^{q}}\\
&\leq \frac{C(u_{0},p,q,\|h\|,B_{1})}{(\rho+1)^{q-1}}.
\end{split}
\end{equation*}
\par
This completes the proof.
\end{proof}

\begin{proposition}\label{pro3}
For any $q>1$, there exist two positive constants $K_{2}$ and $M_{2}$ such that
\begin{equation*}
\begin{split}
\mathbb{P}(\sup\limits_{t\geq 0}[\mathcal{F}^{\psi}_{u}(t+T)-\mathcal{F}^{\psi}_{u}(T)-K_{2}t-M_{2}(\|u(T)\|^{2}+\|u(T)\|^{2p_{1}}+1)]\geq \rho)
\leq CQ_{q}(\rho+1),
\end{split}
\end{equation*}
for all $\rho>0,T\geq 0$ and $u\in H$, where $C$ is a constant depends on $u,q,h,B_{1},B_{2}$.
\end{proposition}

\begin{proof}[Proof of Proposition \ref{pro3}]
By applying It\^{o} formula to $\|\psi u\|^{2}$, it holds that
\begin{equation*}
\begin{split}
d\|\psi u\|^{2}=2(\psi u,\psi(h-uu_{x}-au-u_{xxxx})+\psi_{t} u)dt+\sum\limits_{i=1}^{\infty}b_{i}^{2}\|\psi e_{i}\|^{2}dt
+2\sum\limits_{i=1}^{\infty}b_{i}(\psi u,\psi e_{i})d\beta_{i},
\end{split}
\end{equation*}
this implies that
\begin{equation}\label{12}
\begin{split}
&\|(\psi u)(t+T)\|^{2}-\|(\psi u)(T)\|^{2}+2\int_{T}^{t+T}(\psi^{2} u,u_{xxxx})ds+2a\int_{T}^{t+T}\|\psi u\|^{2}ds\\
=&2\int_{T}^{t+T}(\psi^{2} u,h-uu_{x})ds+2\int_{T}^{t+T}(\psi u,\psi_{t} u)ds+\int_{T}^{t+T}B_{2}^{\psi}(s)ds+M_{2}(t),
\end{split}
\end{equation}
where $B_{2}^{\psi}(t)=\sum\limits_{i=1}^{\infty}b_{i}^{2}\|\psi e_{i}\|^{2},M_{2}(t)=2\int_{T}^{t+T}\sum\limits_{i=1}^{\infty}b_{i}(\psi u,\psi e_{i})d\beta_{i}$
and we have the following estimate
\begin{equation*}
\begin{split}
&B_{2}^{\psi}(t)\leq B_{2},\\
&\langle M_{2}(t)\rangle=4\sum\limits_{i=1}^{\infty}b_{i}^{2}\int_{T}^{t+T}(\psi u,\psi e_{i})^{2}ds\leq 4B_{2}\int_{T}^{t+T}\|\psi u\|^{2}ds.
\end{split}
\end{equation*}
We take $\gamma_{2}=\frac{a}{8B_{2}}$, then
\begin{equation*}
\begin{split}
\gamma_{2}\langle M_{2}(t)\rangle \leq \frac{a}{2}\int_{T}^{t+T}\|\psi u\|^{2}ds.
\end{split}
\end{equation*}
Noting the fact
\begin{equation*}
\begin{split}
(\psi^{2} u,u_{xxxx})=((\psi^{2} u)_{xx},u_{xx})=\int(\psi^{2} u_{xx}^{2}+2\psi_{x}^{2} uu_{xx}+2\psi\psi_{xx}uu_{xx}+4\psi\psi_{x}u_{x}u_{xx})dx,
\end{split}
\end{equation*}
where
\begin{equation*}
\begin{split}
&\int\psi_{x}^{2} uu_{xx}dx\leq C\|u\|_{2}^{2},\\
&\int \psi\psi_{xx}uu_{xx} dx\leq \delta \|\psi u_{xx}\|^{2}+C\|u\|^{2},\\
&\int\psi\psi_{x}u_{x}u_{xx}dx\leq\delta \|\psi u_{xx}\|^{2}+C\|u_{x}\|^{2},
\end{split}
\end{equation*}
thus, we have
\begin{equation*}
\begin{split}
|\int(2\psi_{x}^{2} uu_{xx}+2\psi\psi_{xx}uu_{xx}+4\psi\psi_{x}u_{x}u_{xx})dx|\leq \delta\|\psi u_{xx}\|^{2}+C\|u\|_{2}^{2}.
\end{split}
\end{equation*}
By the same calculations, we have
\begin{equation*}
\begin{split}
(\psi^{2} u,h)&\leq \delta \|\psi u\|^{2}+C_{\delta}\|\psi h\|^{2},\\
(\psi u,\psi_{t} u)&\leq\delta \|\psi u\|^{2}+C_{\delta}\|u\|^{2},\\
|(\psi^{2} u,uu_{x})|&\leq C|(\psi u,\psi_{x} u^{2})|\\
&\leq \delta \|\psi u\|^{2}+C_{\delta}\int u^{4}dx\\
& \leq \delta \|\psi u\|^{2}+C_{\delta}\|u\|_{2}^{\frac{1}{2}}\|u\|^{\frac{7}{2}}\\
& \leq \delta \|\psi u\|^{2}+C_{\delta}(\|u\|_{2}^{2}+\|u\|^{2p_{1}}).
\end{split}
\end{equation*}
Substituting the above estimates in \eqref{12}, by taking $\delta>0$ small enough, we have
\begin{equation}\label{FP1}
\begin{split}
&\|(\psi u)(t+T)\|^{2}-\|(\psi u)(T)\|^{2}+a_{0}\int_{T}^{t+T}(\|\psi u_{xx}\|^{2}+\|\psi u\|^{2})ds\\
\leq&(B_{2}+\|\varphi h\|^{2})t+M_{2}(t)-\gamma_{2}\langle M_{2}(t)\rangle+C\int_{T}^{t+T}\|u(s)\|_{2}^{2}ds+C\int_{T}^{t+T}\|u(s)\|^{2p_{1}}ds.
\end{split}
\end{equation}
\par
Since
\begin{equation*}
\begin{split}
&C\int_{T}^{t+T}\|u\|_{2}^{2}ds\leq C_{2}(\mathcal{E}_{u}(t+T)-\mathcal{E}_{u}(T)+\|u(T)\|^{2}),\\
&C\int_{T}^{t+T}\|u\|^{2p_{1}}ds\leq C_{p_{1}}(\mathcal{E}^{p_{1}}_{u}(t+T)-\mathcal{E}^{p_{1}}_{u}(T)+\|u(T)\|^{2p_{1}}),
\end{split}
\end{equation*}
\eqref{FP1} implies that
\begin{equation*}
\begin{split}
\mathcal{F}^{\psi}_{u}(t+T)-\mathcal{F}^{\psi}_{u}(T)
\leq&(B_{2}+\|\varphi h\|^{2})t+(M_{2}(t)-\gamma_{2}\langle M_{2}(t)\rangle)+C_{2}(\mathcal{E}_{u}(t+T)-\mathcal{E}_{u}(T)+\|u(T)\|^{2})\\
&+C_{p_{1}}(\mathcal{E}^{p_{1}}_{u}(t+T)-\mathcal{E}^{p_{1}}_{u}(T)+\|u(T)\|^{2p_{1}}),
\end{split}
\end{equation*}
namely, we have
\begin{equation*}
\begin{split}
&\mathcal{F}^{\psi}_{u}(t+T)-\mathcal{F}^{\psi}_{u}(T)-[B_{2}+\|\varphi h\|^{2}+C_{2}K_{1}+C_{p_{1}}(L_{p_{1}}+1)]t-C_{2}\|u(T)\|^{2}-C_{p_{1}}\|u(T)\|^{2p_{1}}-2C_{p_{1}}\\
\leq&(M_{2}(t)-\gamma_{2}\langle M_{2}(t)\rangle)+C_{2}(\mathcal{E}_{u}(t+T)-\mathcal{E}_{u}(T)-K_{1}t)\\
&+C_{p_{1}}(\mathcal{E}^{p_{1}}_{u}(t+T)-\mathcal{E}^{p_{1}}_{u}(T)-(L_{p_{1}}+1)t-2),
\end{split}
\end{equation*}
here, we take $K_{2}=B_{2}+\|\varphi h\|^{2}+C_{2}K_{1}+C_{p_{1}}(L_{p_{1}}+1), M_{2}=\max(C_{2},2C_{p_{1}})$. According to Proposition \ref{pro1}, Proposition \ref{pro2} and exponential supermartingale inequality, we can prove the desired result.
\par
This completes the proof.
\end{proof}

\begin{proposition}\label{proM}
For any $q>1$, there exist a $Q_{q}$ and positive constants $\mathcal{K}$ and $\mathcal{M}$ such that
\begin{equation*}
\begin{split}
\mathbb{P}\left(\sup\limits_{t\geq 0}[\mathcal{G}^{\psi}_{u}(t+T)-\mathcal{G}^{\psi}_{u}(T)-\mathcal{K}t-\mathcal{M}(\|u(T)\|^{2}+\|u(T)\|^{2p_{1}}+1)]\geq \rho\right)
\leq CQ_{q}(\rho+1),
\end{split}
\end{equation*}
for all $\rho>0,T\geq 0,$ and $u\in H,$ where $C$ is a constant depends on $u,q,h,B_{1},B_{2}$.
\end{proposition}
\begin{proof}[Proof of Proposition \ref{proM}]
Since $\|\psi u\|_{2}^{2}\leq C(\|\psi u\|^{2}+\|u\|^{2}+\|u_{x}\|^{2}+\|\psi u_{xx}\|^{2})$, we have
\begin{equation*}
\begin{split}
\mathcal{G}^{\psi}_{u}(t+T)-\mathcal{G}^{\psi}_{u}(T)
&\leq \mathcal{E}_{u}(t+T)-\mathcal{E}_{u}(T)+\mathcal{F}^{\psi}_{u}(t+T)-\mathcal{F}^{\psi}_{u}(T)+C\int_{T}^{t+T}(\|u\|^{2}+\|u_{x}\|^{2})ds\\
&\leq C_{3}(\mathcal{E}_{u}(t+T)-\mathcal{E}_{u}(T))+\mathcal{F}^{\psi}_{u}(t+T)-\mathcal{F}^{\psi}_{u}(T)+C_{3}\|u(T)\|^{2}.
\end{split}
\end{equation*}
With the help of Proposition \ref{pro1}, Proposition \ref{pro2}, Proposition \ref{pro3}, we can find suitable $\mathcal{K}$ and $\mathcal{M}$.
\par
This completes the proof.
\end{proof}
\begin{proposition}\label{proM+}
For any $q>1$, there exist a $Q_{q}$ and positive constants $\mathcal{K}$ and $\mathcal{M}$ such that
\begin{equation*}
\begin{split}
\mathbb{P}\left(\sup\limits_{t\geq 0}[\mathcal{E}^{\psi}_{u}(t+T)-\mathcal{E}^{\psi}_{u}(T)-\mathcal{K}t-\mathcal{M}(\|u(T)\|^{2}+\|u(T)\|^{2p_{1}}+1)]\geq \rho\right)
\leq CQ_{q}(\rho+1),
\end{split}
\end{equation*}
for all $\rho>0,T\geq 0,$ and $u\in H,$ where $C$ is a constant depends on $u,q,h,B_{1},B_{2}$.
\end{proposition}
\begin{proof}[Proof of Proposition \ref{proM+}]
Since $\mathcal{E}^{\psi}_{u}(t)=\mathcal{E}^{p_{1}}_{u}(t)+\mathcal{G}^{\psi}_{u}(t)$, with the help of Proposition \ref{pro2} and Proposition \ref{proM}, we can find suitable $\mathcal{K}$ and $\mathcal{M}$.
\par
This completes the proof.
\end{proof}

We introduce the following stopping time
\begin{equation*}
\begin{split}
\tau^{u}:=\inf \{t\geq 0~|~\mathcal{E}^{\psi}_{u}(t)\geq (K+L)t+\rho+M(\|u(0)\|^{2}+\|u(0)\|^{2p_{1}}+1)\}.
\end{split}
\end{equation*}

\begin{proposition}\label{proT1}
For any $q>1$, there exists a $Q_{q}$ such that if $K\geq \mathcal{K}$ and $M\geq 1+\mathcal{M}$
\begin{equation}\label{14}
\begin{split}
\mathbb{P}(l\leq\tau^{u}<+\infty)\leq CQ_{q}(\rho+Ll+1),
\end{split}
\end{equation}
for all $L,l\geq0,$ and $u\in H,$ where $C$ is a constants depends on $u,q,h,B_{1},B_{2}$.
\end{proposition}
\begin{proof}[Proof of Proposition \ref{proT1}]
On the event $\{l\leq\tau^{u}<\infty\},$ the definition of $\tau^{u}$ implies that
\begin{equation*}
\begin{split}
\mathcal{E}^{\psi}_{u}(\tau^{u})
&\geq (K+L)\tau^{u}+\rho+M(\|u(0)\|^{2}+\|u(0)\|^{2p_{1}}+1)\\
&\geq \mathcal{K}\tau^{u}+Ll+\|u(0)\|^{2}+\|u(0)\|^{2p_{1}}+\mathcal{M}(\|u(0)\|^{2}+\|u(0)\|^{2p_{1}}+1)+\rho\\
&\geq \mathcal{E}^{\psi}_{u}(0)+\mathcal{K}\tau^{u}+\mathcal{M}(\|u(0)\|^{2}+\|u(0)\|^{2p_{1}}+1)+Ll+\rho,
\end{split}
\end{equation*}
thus, we have
\begin{equation*}
\begin{split}
\sup\limits_{t\geq 0}[\mathcal{E}^{\psi}_{u}(t)-\mathcal{E}^{\psi}_{u}(0)-\mathcal{K}t-\mathcal{M}(\|u(0)\|^{2}+\|u(0)\|^{2p_{1}}+1)]\geq Ll+\rho,
\end{split}
\end{equation*}
with the help of Proposition \ref{proM}, we have \eqref{14}.
\par
This completes the proof.
\end{proof}

\section{Foia\c{s}-Prodi estimate }
Foia\c{s}-Prodi estimate was firstly established in \cite{FPE}, now, it is a powerful tool to
establish the ergodicity for SPDEs, and it is often used to compensate the degeneracy of the noise, see \cite{KS12, PMJEE2005, PMJEE2008, N2} and the references therein.
Now, we establish Foia\c{s}-Prodi estimate of hyperviscous Burgers equation on the whole line.

\subsection{Foia\c{s}-Prodi estimate }
Let us consider the following equations
\begin{equation*}
\begin{split}
&u_{t}+au+u_{xxxx}+uu_{x}=h+\eta ~~{\rm{in}}~\mathbb{R},\\
&v_{t}+av+v_{xxxx}+vv_{x}+P_{N}(uu_{x}-vv_{x})=h+\eta ~~{\rm{in}}~\mathbb{R}.
\end{split}
\end{equation*}
Let $w=u-v,$ then we have
\begin{theorem}\label{FP}
For $w$, we have the following estimates.
\par
(1) There exists a constant $C>0$ such that
\begin{equation*}
\begin{split}
\|w(t)\|_{1}^{2}\leq \|w(s)\|_{1}^{2}\exp\left (-a(t-s)+C\int_{s}^{t}(\|u\|_{1}^{2}+\|v\|_{1}^{2})dr\right),
\end{split}
\end{equation*}
for any $t\geq s\geq0$.
\par
(2) For any $\varepsilon>0,$ there exists a time $T>0$ and an integer $N\geq1$ such that
\begin{equation*}
\begin{split}
\|w(t+T)\|_{1}^{2}\leq \|w(s+T)\|_{1}^{2}\exp\left (-a(t-s)+C_{\ast}\varepsilon \int_{s+T}^{t+T}(\|u\|_{2}^{2}+\|v\|_{2}^{2}+\|\psi u\|_{1}^{2}+\|\psi v\|_{1}^{2})dr\right)
\end{split}
\end{equation*}
for any $t\geq s\geq0,$ where $C_{\ast}>0$ is a constant depending on $a$.

\end{theorem}

\begin{proof}[Proof of Theorem \ref{FP}]
We can see that $w$ satisfies that
\begin{equation}\label{WE}
\begin{split}
w_{t}+aw+w_{xxxx}+Q_{N}(uu_{x}-vv_{x})=0,
\end{split}
\end{equation}
where $Q_{N}:=Id-P_{N}$. Taking the $L^{2}-$inner product of \eqref{WE} with $w,$ we get
\begin{equation*}
\begin{split}
\frac{1}{2}\partial_{t}\|w\|^{2}+a\|w\|^{2}+\|w_{xx}\|^{2}+(Q_{N}(uu_{x}-vv_{x}),w)=0.
\end{split}
\end{equation*}
Taking the $L^{2}-$inner product of \eqref{WE} with $-w_{xx},$ we get
\begin{equation*}
\begin{split}
\frac{1}{2}\partial_{t}\|w_{x}\|^{2}+a\|w_{x}\|^{2}+\|w_{xxx}\|^{2}+(Q_{N}(uu_{x}-vv_{x}),-w_{xx})=0,
\end{split}
\end{equation*}
the above facts imply that
\begin{equation*}
\begin{split}
\frac{1}{2}\partial_{t}\|w\|_{1}^{2}+a\|w\|_{1}^{2}+(\|w_{xxx}\|^{2}+\|w_{xx}\|^{2})+(Q_{N}(uu_{x}-vv_{x}),w-w_{xx})=0.
\end{split}
\end{equation*}
\par
(1) By using the Cauchy inequality, we have
\begin{equation*}
\begin{split}
-(Q_{N}(uu_{x}-vv_{x}),w-w_{xx})\leq C(\|u\|_{1}+\|v\|_{1})\|w\|_{1}\|w\|_{2}\leq \frac{a_{0}}{4}\|w\|_{2}^{2}+C(\|u\|_{1}^{2}+\|v\|_{1}^{2})\|w\|_{1}^{2},
\end{split}
\end{equation*}
thus, Gronwall inequality implies (1).

\par
(2) Let us first recall the following lemma
\begin{lemma}\cite[Lemma 2.1]{N2}\label{L1}
Let us fix any $s>0.$ For any $\varepsilon,A>0,$ there is an integer $N\geq1$ such that
\begin{equation*}
\begin{split}
\|Q_{N}\chi_{A}f\|\leq \varepsilon\|f\|_{s},~~\forall f\in H^{s},
\end{split}
\end{equation*}
where $Q_{N}:=Id-P_{N}$ and $\chi_{A}$ is any smooth cut-off function from $\mathbb{R}$ to $[0,1]$
such that
\begin{equation*}
\begin{array}{l}
\chi_{A}=\left\{
\begin{array}{lll}
1\\
0
\end{array}
\right.
\end{array}
\begin{array}{lll}
x\in [-\frac{A}{2},\frac{A}{2}],
\\x\in [-A,A]^{c}.
\end{array}
\end{equation*}
\end{lemma}
We split $(Q_{N}(uu_{x}-vv_{x}),w)$ into two parts,
\begin{equation*}
\begin{split}
(Q_{N}(uu_{x}-vv_{x}),w)&=(Q_{N}\chi_{A}(uu_{x}-vv_{x}),w)+(Q_{N}(1-\chi_{A})(uu_{x}-vv_{x}),w)\\
&=:I_{1}+I_{2}.
\end{split}
\end{equation*}
Then, applying Lemma \ref{L1} with $s=1$ to $I_{1}$, we have
\begin{equation*}
\begin{split}
I_{1}&=(Q_{N}\chi_{A}(uu_{x}-vv_{x}),w)\\
&\leq \|w\|\|Q_{N}\chi_{A}(uu_{x}-vv_{x})\|\\
&\leq \varepsilon\|w\|\|uu_{x}-vv_{x}\|_{1}\\
&\leq \varepsilon\|w\|\|u^{2}-v^{2}\|_{2}\\
&\leq \varepsilon\|w\|\|w\|_{2}\|u+v\|_{2}\\
&\leq \varepsilon C_{\delta}\|w\|^{2}(\|u\|_{2}^{2}+\|v\|_{2}^{2})+\delta\|w\|_{2}^{2}.
\end{split}
\end{equation*}
For $I_{2}$, it follows from the property of the weight function $\psi$ that
\begin{equation*}
\begin{split}
\psi^{-1}(x,t)\leq \varepsilon,~~\forall t\geq T, |x|\geq\frac{A}{2},
\end{split}
\end{equation*}
provided that $T,A>0$ are large enough. This implies that
\begin{equation*}
\begin{split}
I_{2}&=((1-\chi_{A})(uu_{x}-vv_{x}),Q_{N}w)\\
&=(Q_{N}w,(1-\chi_{A})(\psi uu_{x}-\psi vv_{x})\psi^{-1})\\
&\leq \varepsilon\|w\|\|\psi uu_{x}-\psi vv_{x}\|.
\end{split}
\end{equation*}
Since
\begin{equation*}
\begin{split}
\|\psi uu_{x}-\psi vv_{x}\|
&=\|w\psi u_{x}+\psi vw_{x}\|\\
&\leq \|w\|_{1}\|\psi u_{x}\|+\|\psi v\|_{1}\|w\|_{1}\\
&= \|w\|_{1}(\|(\psi u)_{x}-\psi_{x} u\|+\|\psi v\|_{1})\\
&\leq C\|w\|_{1}(\|\psi u\|_{1}+\|\psi v\|_{1}+\|u\|),
\end{split}
\end{equation*}
we have
\begin{equation*}
\begin{split}
I_{2}&\leq \varepsilon C\|w\|\|w\|_{1}(\|\psi u\|_{1}+\|\psi v\|_{1}+\|u\|)\\
&\leq \varepsilon C_{\delta}\|w\|^{2}(\|\psi u\|_{1}^{2}+\|\psi v\|_{1}^{2}+\|u\|^{2})+\delta\|w\|_{1}^{2},
\end{split}
\end{equation*}
by taking $\delta>0$ small enough, it holds that
\begin{equation*}
\begin{split}
\partial_{t}\|w\|^{2}+[a-\varepsilon C(\|u\|_{2}^{2}+\|v\|_{2}^{2}+\|\psi u\|_{1}^{2}+\|\psi v\|_{1}^{2})]\|w\|^{2}\leq0.
\end{split}
\end{equation*}
\par
On the other hand,
\begin{equation*}
\begin{split}
(Q_{N}(uu_{x}-vv_{x}),w_{xx})&=(Q_{N}\chi_{A}(uu_{x}-vv_{x}),w_{xx})+(Q_{N}(1-\chi_{A})(uu_{x}-vv_{x}),w_{xx})\\
&=:I_{3}+I_{4}.
\end{split}
\end{equation*}
Then, applying Lemma \ref{L1} with $s=1$ to $I_{3}$,
\begin{equation*}
\begin{split}
I_{3}&=(Q_{N}\chi_{A}(uu_{x}-vv_{x}),w_{xx})\\
&\leq \|w_{xx}\|\|Q_{N}\chi_{A}(uu_{x}-vv_{x})\|\\
&\leq \varepsilon\|w_{xx}\|\|uu_{x}-vv_{x}\|_{1}\\
&\leq \varepsilon\|w_{xx}\|\|u^{2}-v^{2}\|_{2}\\
&\leq \varepsilon\|w_{xx}\|\|w\|_{2}\|u+v\|_{2}\\
&\leq \varepsilon\|w\|_{2}^{2}\|u+v\|_{2}\\
&\leq \varepsilon\|w\|_{1}\|w\|_{3}\|u+v\|_{2}\\
&\leq \varepsilon C_{\delta}\|w\|_{1}^{2}(\|u\|_{2}^{2}+\|v\|_{2}^{2})+\delta\|w\|_{3}^{2}\\
\end{split}
\end{equation*}
and
\begin{equation*}
\begin{split}
I_{4}&=((1-\chi_{A})(uu_{x}-vv_{x}),Q_{N}w_{xx})\\
&=(Q_{N}w_{xx},(1-\chi_{A})(\psi uu_{x}-\psi vv_{x})\psi^{-1})\\
&\leq \varepsilon\|w_{xx}\|\|\psi uu_{x}-\psi vv_{x}\|\\
&\leq \varepsilon C\|w_{xx}\|\|w\|_{1}(\|\psi u\|_{1}+\|\psi v\|_{1}+\|u\|)\\
&\leq \varepsilon C_{\delta}\|w\|_{1}^{2}(\|\psi u\|_{1}^{2}+\|\psi v\|_{1}^{2}+\|u\|^{2})+\delta\|w_{xx}\|^{2}.
\end{split}
\end{equation*}
By taking $\delta>0$ small enough, we have
\begin{equation*}
\begin{split}
\partial_{t}\|w\|_{1}^{2}+[a-\varepsilon C(\|u\|_{2}^{2}+\|v\|_{2}^{2}+\|\psi u\|_{1}^{2}+\|\psi v\|_{1}^{2})]\|w\|_{1}^{2}\leq0,
\end{split}
\end{equation*}
thus, Gronwall inequality implies (2).
\par
This completes the proof.
\end{proof}

\subsection{Growth estimate for the auxiliary process}
\begin{proposition}\label{pro10}
If \eqref{41} holds for some $N\geq 1$, then there exist constants $\kappa,C>0$ such that
\begin{equation*}
\begin{split}
\mathbb{P}(\tau^{v}<\infty)
\leq CQ_{q}(\rho+1)+\frac{1}{2}\left(\exp(Ce^{C(\rho+\|u\|^{2p_{1}}+\|u^{\prime}\|^{2p_{1}}+1)}d^{2})-1 \right)^{\frac{1}{2}}
\end{split}
\end{equation*}
for any $\rho>0$ and $u,u^{\prime}\in V$ with $d=\|u-u^{\prime}\|_{1}.$
\end{proposition}

\begin{proof}[Proof of Proposition \ref{pro10}]
The proof is divided into the following steps.
\par
\textbf{Step 1}. Let $\tau=\tau^{v}\wedge\tau^{u}\wedge\tau^{u^{\prime}}.$ We introduce the truncated processes $\{\hat{u}(t)\}_{t\geq0},$ $\{\hat{u}^{\prime}(t)\}_{t\geq0}$
and $\{\hat{v}(t)\}_{t\geq0}$ as follows:
\begin{equation*}
\begin{split}
\hat{u}(t)=
\left\{
\begin{array}{lll}
u(t),~{\rm{for}}~0\leq t\leq \tau,\\
{\rm{solves}}~ z_{t}+Az=0,~{\rm{for}}~t\geq \tau,
\end{array}
\right.
\\
\hat{u}^{\prime}(t)=
\left\{
\begin{array}{lll}
u^{\prime}(t),~{\rm{for}}~0\leq t\leq \tau,\\
{\rm{solves}}~ z_{t}+Az=0,~{\rm{for}}~t\geq \tau,
\end{array}
\right.
\\
\hat{v}(t)=
\left\{
\begin{array}{lll}
v(t),~{\rm{for}}~0\leq t\leq \tau,\\
{\rm{solves}}~ z_{t}+Az=0,~{\rm{for}}~t\geq \tau.
\end{array}
\right.
\end{split}
\end{equation*}
According to Proposition \ref{proT1} with $l=0$, we have
\begin{equation*}
\begin{split}
\mathbb{P}(\tau^{v}<\infty)\leq CQ_{q}(\rho+1)+\mathbb{P}(\tau^{\hat{v}}<\infty).
\end{split}
\end{equation*}
\par
\textbf{Step 2}. Without loss of generality, we assume that the underlying
probability space $(\Omega,\mathcal{F},\mathbb{P})$ is of a particular form $\Omega:=C_{0}([0,+\infty);H)$ is the space of all continuous functions taking values in $H$ and vanishing at $t=0,$ $\mathbb{P}$
is the distribution of the Wiener process
$$\xi(t)=\sum_{i=1}^{\infty}b_{i}\beta_{i}(t)e_{i},$$
$\mathcal{F}$ is the completion of the Borel $\sigma$-algebra of associated with the topology
of uniform convergence on every compact set. For any integer $N\geq1,$ we define the transform
\begin{equation*}
\begin{split}
\Phi^{u,u^{\prime}}:&~\Omega\longrightarrow \Omega\\
&\omega(t)\rightarrow \omega(t)-\int_{0}^{t}\mathbb{I}_{s\leq \tau}P_{N}[\hat{u}\hat{u}_{x}-\hat{v}\hat{v}_{x}]ds.
\end{split}
\end{equation*}
Due to the pathwise uniqueness for the stochastic hyperviscous Burgers equation, we have
\begin{equation*}
\begin{split}
\mathbb{P}\{\hat{u}^{\prime}(\Phi^{u,u^{\prime}}(\omega),t)=\hat{v}(\omega,t),~\forall t\geq0\}=1,
\end{split}
\end{equation*}
this leads to
\begin{equation*}
\begin{split}
\mathbb{P}(\tau^{\hat{v}}<\infty)=\Phi^{u,u^{\prime}}_{\ast}\mathbb{P}(\tau^{\hat{u}^{\prime}}<\infty)\leq \mathbb{P}(\tau^{\hat{u}^{\prime}}<\infty)+\|\mathbb{P}-\Phi^{u,u^{\prime}}_{\ast}\mathbb{P}\|_{var}.
\end{split}
\end{equation*}
An estimate for $\mathbb{P}(\tau^{\hat{u}^{\prime}}<\infty)$ is proved in Proposition \ref{proT2}.
\par
We express $\Omega=C([0,+\infty);P_{N}H)\oplus C([0,+\infty);Q_{N}H),$ for $\omega=(\omega^{(1)},\omega^{(2)})\in \Omega,$
\begin{equation*}
\begin{split}
\Phi^{u,u^{\prime}}(\omega)=\Phi^{u,u^{\prime}}(\omega^{(1)},\omega^{(2)})=(\Psi^{u,u^{\prime}}(\omega^{(1)},\omega^{(2)}),\omega^{(2)}),
\end{split}
\end{equation*}
where
\begin{equation*}
\begin{split}
\Psi^{u,u^{\prime}}:&~\Omega\longrightarrow C([0,+\infty);P_{N}H)\\
&\omega(t)\rightarrow \omega^{(1)}(t)+\int_{0}^{t}A(s,\omega^{(1)},\omega^{(2)})ds
\end{split}
\end{equation*}
with $$A(s):=-\mathbb{I}_{s\leq \tau}P_{N}[\hat{u}\hat{u}_{x}-\hat{v}\hat{v}_{x}].$$ By applying \cite[Lemma 3.3.13]{KS12}, we have
\begin{equation*}
\begin{split}
\|\mathbb{P}-\Phi^{u,u^{\prime}}_{\ast}\mathbb{P}\|_{var}\leq \int_{C([0,+\infty);Q_{N}H)}\|\mathbb{P}_{N}-\Psi^{u,u^{\prime}}_{\ast}(\mathbb{P}_{N},\omega^{(2)})\|_{var}\mathbb{P}_{N}^{\bot}(d\omega^{(2)}).
\end{split}
\end{equation*}
If for any $C>0$ and $\omega^{(2)},$ it holds that
\begin{equation}\label{10}
\begin{split}
\mathbb{E}_{N}\exp(C\int_{0}^{\infty}\|A(t)\|^{2}dt)<\infty,
\end{split}
\end{equation}
then Girsanov theorem shows that
\begin{equation*}
\begin{split}
\|\mathbb{P}_{N}-\Psi^{u,u^{\prime}}_{\ast}(\mathbb{P}_{N},\omega^{(2)})\|_{var}
\leq \frac{1}{2}[(\mathbb{E}_{N}\exp(6\sup_{1\leq i\leq N}b_{i}^{-2}\int_{0}^{\infty}\|A(t)\|^{2}dt))^{\frac{1}{2}}-1]^{\frac{1}{2}}.
\end{split}
\end{equation*}
\par
\textbf{Step 3}. Let us estimate $w(t):=\hat{u}(t)-\hat{v}(t)$ for $0\leq t\leq\tau$, for any $t\in [0,\tau),$ it holds that
\begin{equation}\label{11}
\begin{split}
\|w(t)\|_{1}^{2}\leq C\exp\left (-\frac{a}{2}t+C(\rho+\|u\|^{2p_{1}}+\|u^{\prime}\|^{2p_{1}}+1)\right)d^{2}.
\end{split}
\end{equation}
\par
Indeed, noting the fact $w(t)=u(t)-v(t)$ for $0\leq t\leq\tau$ and the definition of $\tau$,
we have
\begin{equation}\label{9}
\begin{split}
&\mathcal{E}_{\hat{u}}^{\psi}(t)=\mathcal{E}_{u}^{\psi}(t)\leq (K+L)t+\rho+M(\|u\|^{2}+\|u\|^{2p_{1}}+1),\\
&\mathcal{E}_{\hat{v}}^{\psi}(t)=\mathcal{E}_{v}^{\psi}(t)\leq (K+L)t+\rho+M(\|u^{\prime}\|^{2}+\|u^{\prime}\|^{2p_{1}}+1),
\end{split}
\end{equation}
for $0\leq t\leq\tau$. Now, we distinguish the following two cases.
\par
\textit{Case 1. $\tau\leq T.$}
\par
Foia\c{s}-Prodi estimate in Theorem \ref{FP} shows that
\begin{equation*}
\begin{split}
\|w(t)\|_{1}^{2}&\leq \|w(0)\|_{1}^{2}\exp\left (-at+C\int_{0}^{t}(\|u\|_{1}^{2}+\|v\|_{1}^{2})dr\right)\\
&\leq \|w(0)\|_{1}^{2}\exp\left(-at+C((K+L)T+\rho+M(\|u\|^{2}+\|u\|^{2p_{1}}+\|u^{\prime}\|^{2}+\|u^{\prime}\|^{2p_{1}}+2))\right),
\end{split}
\end{equation*}
this implies that
\begin{equation}\label{8}
\begin{split}
\|w(t)\|_{1}^{2}\leq C\exp\left(-at+C(\rho+\|u\|^{2p_{1}}+\|u^{\prime}\|^{2p_{1}}+1)\right)d^{2}
\end{split}
\end{equation}
\par
\textit{Case 2. $\tau> T.$}
\par
We can see that \eqref{8} holds for $w$ on $[0,T]$. We apply Foia\c{s}-Prodi estimate in Theorem \ref{FP} with $\varepsilon=\frac{aa_{0}}{4C_{\ast}(K+L)}$ to $w$ on $[T,\tau)$,
it holds that
\begin{equation*}
\begin{split}
\|w(t)\|_{1}^{2}\leq \|w(T)\|_{1}^{2}\exp\left (-a(t-T)+C_{\ast}\varepsilon \int_{T}^{t}(\|u\|_{2}^{2}+\|v\|_{2}^{2}+\|\psi u\|_{1}^{2}+\|\psi v\|_{1}^{2})dr\right).
\end{split}
\end{equation*}
\eqref{9} leads to
\begin{equation*}
\begin{split}
\|w(t)\|_{1}^{2}\leq C\exp\left (-\frac{a}{2}t+C(\rho+\|u\|^{2p_{1}}+\|u^{\prime}\|^{2p_{1}}+1)\right)d^{2}.
\end{split}
\end{equation*}
Case 1 and Case 2 imply \eqref{11}.
\par
\textbf{Step 4}. Let us verify the Novikov condition \eqref{10}.
\par
Indeed, since $\|A(t)\|\leq C\cdot 1_{t\leq \tau}\cdot\|w\|_{1}(\|u\|_{1}+\|v\|_{1})$, by integrating by parts and \eqref{11}, we have
\begin{equation*}
\begin{split}
\int_{0}^{\infty}\|A(t)\|^{2}dt
&=\int_{0}^{\tau}\|A(t)\|^{2}dt\\
&\leq C\int_{0}^{\tau}\|w\|^{2}_{1}(\|u\|^{2}_{1}+\|v\|^{2}_{1})dt\\
&\leq C\exp\left (C(\rho+\|u\|^{2p_{1}}+\|u^{\prime}\|^{2p_{1}}+1)\right) d^{2}\int_{0}^{\tau}\exp\left (-\frac{a}{2}t \right)(\|u\|^{2}_{1}+\|v\|^{2}_{1})dt\\
&=C\exp\left (C(\rho+\|u\|^{2p_{1}}+\|u^{\prime}\|^{2p_{1}}+1)\right) d^{2}\int_{0}^{\tau}\exp\left (-\frac{a}{2}t\right)\textit{d}(\int_{0}^{t}(\|u\|^{2}_{1}+\|v\|^{2}_{1})ds)\\
&\leq C\exp\left (C(\rho+\|u\|^{2p_{1}}+\|u^{\prime}\|^{2p_{1}}+1)\right) d^{2},
\end{split}
\end{equation*}
namely, we prove Novikov condition \eqref{10}. Then, we can obtain
\begin{equation}\label{13}
\begin{split}
\|\mathbb{P}-\Phi^{u,u^{\prime}}_{\ast}\mathbb{P}\|_{var}\leq \frac{1}{2}\left(\exp(Ce^{C(\rho+\|u\|^{2p_{1}}+\|u^{\prime}\|^{2p_{1}}+1)}d^{2})-1 \right)^{\frac{1}{2}}.
\end{split}
\end{equation}
\par
This completes the proof.
\end{proof}

\subsection{Weighted estimates for the truncated process}
\begin{proposition}\label{proM2}
For any $q>1$, there exists a $Q_{q}$ and positive constants $K_{4}$ and $M_{4}$ such that
\begin{equation*}
\begin{split}
\mathbb{P}\left(\sup\limits_{t\geq 0}[\mathcal{E}^{\psi}_{\hat{u}^{\prime}}(t)-K_{4}t-M_{4}(\|u^{\prime}\|^{2}+\|u^{\prime}\|^{2p_{1}}+1)]\geq \rho\right)
\leq CQ_{q}(\rho+1),
\end{split}
\end{equation*}
for all $\rho>0,$ and $u^{\prime}\in H,$ where $C$ is a constant depends on $u^{\prime},q,h,B_{1},B_{2}$.
\end{proposition}
\begin{proof}[Proof of Proposition \ref{proM2}]
Let us first consider the auxiliary equation
\begin{equation*}
\begin{split}
z_{t}+az+z_{xxxx}=0 ~~{\rm{in}}~\mathbb{R}.
\end{split}
\end{equation*}
\par
On the one hand, by the similar arguments as in \eqref{4}, for any $t,T\geq0,$ we have
\begin{equation*}
\begin{split}
\|z(t+T)\|^{2p_{1}}+p_{1}a\int_{T}^{t+T}\|z(s)\|^{2p_{1}}ds\leq \|z(T)\|^{2p_{1}}.
\end{split}
\end{equation*}
It follows from the definition of $\mathcal{E}_{z}^{p}$ that
\begin{equation}\label{17}
\begin{split}
\mathcal{E}_{z}^{p_{1}}(t+T)\leq C\mathcal{E}_{z}^{p_{1}}(T)~~{\rm{for}}~t,T\geq0.
\end{split}
\end{equation}
\par
On the other hand, for any $t,T\geq0,$ we have
\begin{equation*}
\begin{split}
\|z(t+T)\|^{2}+a_{0}\int_{T}^{t+T}\|z(s)\|_{2}^{2}ds\leq \|z(T)\|^{2}.
\end{split}
\end{equation*}
By the similar arguments as in Proposition \ref{pro3}, it holds that
\begin{equation*}
\begin{split}
\|(\psi z)(t+T)\|^{2}+a_{0}\int_{T}^{t+T}(\|(\psi z)(s)\|^{2}+\|(\psi z_{xx})(s)\|^{2})ds\leq C(\|(\psi z)(T)\|^{2}+\|z(T)\|^{2}).
\end{split}
\end{equation*}
Since $\|\psi z\|_{2}^{2}\leq C(\|\psi z\|^{2}+\|z\|^{2}+\|z_{x}\|^{2}+\|\psi z_{xx}\|^{2})$, we have
\begin{equation*}
\begin{split}
\|z(t+T)\|^{2}+\|(\psi z)(t+T)\|^{2}+a_{0}\int_{T}^{t+T}(\|z(s)\|_{2}^{2}+\|(\psi z)(s)\|_{2}^{2})ds\leq C(\|z(T)\|^{2}+\|(\psi z)(T)\|^{2}),
\end{split}
\end{equation*}
this implies that
\begin{equation}\label{18}
\begin{split}
\mathcal{G}_{z}^{\psi}(t+T)\leq C\mathcal{G}_{z}^{\psi}(T)~~{\rm{for}}~t,T\geq0.
\end{split}
\end{equation}
With the help of $\mathcal{E}^{\psi}_{u}(t)=\mathcal{E}^{p_{1}}_{u}(t)+\mathcal{G}^{\psi}_{u}(t)$, \eqref{17} and \eqref{18}, we have
\begin{equation*}
\begin{split}
\mathcal{E}_{z}^{\psi}(t+T)\leq C\mathcal{E}_{z}^{\psi}(T)~~{\rm{for}}~t,T\geq0.
\end{split}
\end{equation*}
This implies that on the event $\{\tau<\infty\},$ it holds that
\begin{equation*}
\begin{split}
\mathcal{E}_{\hat{u}^{\prime}}^{\psi}(t+\tau)\leq C\mathcal{E}_{u^{\prime}}^{\psi}(\tau)~~{\rm{for}}~t\geq0,
\end{split}
\end{equation*}
then, on the event $\{\tau<\infty\}$, there exists a constant $C^{\prime}>1$ such that
\begin{equation*}
\begin{split}
\mathcal{E}_{\hat{u}^{\prime}}^{\psi}(t)-C^{\prime}\mathcal{K}t\leq C^{\prime}\sup\limits_{t\geq0}(\mathcal{E}_{u^{\prime}}^{\psi}(t)-\mathcal{K}t)~~{\rm{for}}~t\geq0.
\end{split}
\end{equation*}
On the event $\{\tau=+\infty\}$, it is easy to see that
\begin{equation*}
\begin{split}
\mathcal{E}_{\hat{u}^{\prime}}^{\psi}(t)-C^{\prime}\mathcal{K}t\leq \sup\limits_{t\geq0}(\mathcal{E}_{u^{\prime}}^{\psi}(t)-\mathcal{K}t)~~{\rm{for}}~t\geq0,
\end{split}
\end{equation*}
this implies that
\begin{equation*}
\begin{split}
&\mathcal{E}_{\hat{u}^{\prime}}^{\psi}(t)-C^{\prime}\mathcal{K}t-C^{\prime}(\mathcal{M}+1)(\|u^{\prime}\|^{2}+\|u^{\prime}\|^{2p_{1}}+1))\\
\leq&\mathcal{E}_{\hat{u}^{\prime}}^{\psi}(t)-C^{\prime}\mathcal{K}t-C^{\prime}(\|u^{\prime}\|^{2}+\|u^{\prime}\|^{2p_{1}}+\mathcal{M}(\|u^{\prime}\|^{2}+\|u^{\prime}\|^{2p_{1}}+1))\\
\leq &C^{\prime}\sup\limits_{t\geq0}(\mathcal{E}_{u^{\prime}}^{\psi}(t)-\mathcal{K}t-\mathcal{E}_{u^{\prime}}^{\psi}(0)-\mathcal{M}(\|u^{\prime}\|^{2}+\|u^{\prime}\|^{2p_{1}}+1))~~{\rm{for}}~t\geq0.
\end{split}
\end{equation*}
By taking $K_{4}=C^{\prime}\mathcal{K}, M_{4}=C^{\prime}(\mathcal{M}+1),$ with the help of Proposition \ref{proM+}, we can prove Proposition \ref{proM2}.
\par
This completes the proof.
\end{proof}

\begin{proposition}\label{proT2}
For any $q>1$, there exists a $Q_{q}$ such that if $K\geq K_{4}$ and $M\geq M_{4}$
\begin{equation}\label{15}
\begin{split}
\mathbb{P}(l\leq\tau^{\hat{u}^{\prime}}<+\infty)\leq CQ_{q}(\rho+Ll+1),
\end{split}
\end{equation}
for all $L,l\geq0,\rho>0,$ and $u\in H,$ where $C$ is a constants depends on $u,q,h,B_{1},B_{2}$.
\end{proposition}
\begin{proof}[Proof of Proposition \ref{proT2}]
On the event $\{l\leq\tau^{\hat{u}^{\prime}}<\infty\},$ the definition of $\tau^{\hat{u}^{\prime}}$ implies that
\begin{equation*}
\begin{split}
\mathcal{E}^{\psi}_{\hat{u}^{\prime}}(\tau^{\hat{u}^{\prime}})
&\geq (K+L)\tau^{\hat{u}^{\prime}}+\rho+M(\|u^{\prime}\|^{2}+\|u^{\prime}\|^{2p_{1}}+1)\\
&\geq K_{4}\tau^{\hat{u}^{\prime}}+Ll+M_{4}(\|u^{\prime}\|^{2}+\|u^{\prime}\|^{2p_{1}}+1)+\rho\\
&\geq K_{4}\tau^{\hat{u}^{\prime}}+M_{4}(\|u^{\prime}\|^{2}+\|u^{\prime}\|^{2p_{1}}+1)+Ll+\rho,
\end{split}
\end{equation*}
thus, we have
\begin{equation*}
\begin{split}
\sup\limits_{t\geq 0}[\mathcal{E}^{\psi}_{\hat{u}^{\prime}}(t)-K_{4}t-M_{4}(\|u^{\prime}\|^{2}+\|u^{\prime}\|^{2p_{1}}+1)]\geq Ll+\rho,
\end{split}
\end{equation*}
with the help of Proposition \ref{proM2}, we have \eqref{15}.
\par
This completes the proof.
\end{proof}

\section{Proof of Theorem \ref{MT1}}
With the help of Theorem \ref{THACPM}, in order to prove Theorem \ref{MT1}, we only need to prove recurrence and
polynomial mixing squeezing.
\subsection{Recurrence}
\begin{proposition}\label{proI}
For any $R,d>0,$ there exist constants $p,T>0$ and an integer $N\geq1$ such that
\begin{equation}
\begin{split}
\mathbb{P}(u(T)\in B_{V}(0,d))\geq p
\end{split}
\end{equation}
holds for any $u_{0}\in B_{V}(0,R)$, provided that \eqref{41} holds.
\end{proposition}
\begin{proof}[Proof of Proposition \ref{proI}]
Let us take $T,\delta>0$ and denote $y_{1}(t):=th+W(t)$ and
\begin{equation*}
\begin{split}
\Gamma_{\delta}:=\left\{\sup\limits_{t\in [0,T]}\|y_{1}(t)\|_{3}\leq \delta \right\}.
\end{split}
\end{equation*}
Let $y_{2}(t):=u(t)-y_{1}(t).$ The proof is divided into the following steps.
\par
\textit{Step 1.} For any $T>0$, there exists sufficiently small $0<\delta<1$ such that $\|y_{2}(t)\|$
is uniformly bounded on $[0,T]\times \Gamma_{\delta}.$ We can see that $y_{2}$ satisfies that
\begin{equation*}
\begin{array}{l}
\left\{
\begin{array}{lll}
\partial_{t}y_{2}+ay_{2}+y_{2xxxx}=-ay_{1}-y_{1xxxx}-uu_{x}\\
y_{2}(x,0)=u_{0}
\end{array}
\right.
\end{array}
\begin{array}{lll}
\textrm{in}~\mathbb{R},
\\\textrm{in}~\mathbb{R},
\end{array}
\end{equation*}
Indeed, multiplying the above equation by $y_{2}-y_{2xx}$ and then performing integration by parts over $\mathbb{R}$, we get
\begin{equation*}
\begin{split}
&\frac{1}{2}\partial_{t}(\|y_{2}\|^{2}+\|y_{2x}\|^{2})+a(\|y_{2}\|^{2}+\|y_{2x}\|^{2})+(\|y_{2xx}\|^{2}+\|y_{2xxx}\|^{2})\\
=&-a(y_{1},y_{2}-y_{2xx})-(y_{1xxxx},y_{2}-y_{2xx})-(uu_{x},y_{2}-y_{2xx}).
\end{split}
\end{equation*}
On the event $\Gamma_{\delta},$ we have
\begin{equation*}
\begin{split}
&-a(y_{1},y_{2}-y_{2xx})\leq C\|y_{1}\|^{2}+\varepsilon \|y_{2}\|_{2}^{2}\leq C\delta^{2}+\varepsilon \|y_{2}\|_{2}^{2},\\
&-(y_{1xxxx},y_{2})=-(y_{1xx},y_{2xx})\leq C\|y_{1}\|_{2}^{2}+\varepsilon \|y_{2}\|_{2}^{2}\leq C\delta^{2}+\varepsilon \|y_{2}\|_{2}^{2},\\
&-(y_{1xxxx},-y_{2xx})=(y_{1xxx},y_{2xxx})\leq C\|y_{1xxx}\|^{2}+\varepsilon \|y_{2}\|_{3}^{2}\leq C\delta^{2}+\varepsilon \|y_{2}\|_{3}^{2},\\
&-(uu_{x},y_{2})=-((y_{1}+y_{2})(y_{1}+y_{2})_{x},y_{2})=-(y_{1}y_{1x}+y_{1}y_{2x}+y_{1x}y_{2},y_{2})\\
&~~~~~~~~~~~~~~~=-(y_{1}y_{1x},y_{2})+\frac{1}{2}(y_{1x},y_{2}^{2})\leq C\delta(1+\|y_{2}\|^{2}),\\
&-(uu_{x},-y_{2xx})=((y_{1}+y_{2})(y_{1}+y_{2})_{x},y_{2xx})=(y_{1}y_{1x}+y_{1}y_{2x}+y_{1x}y_{2}+y_{2}y_{2x},y_{2xx}),\\
&(y_{1}y_{1x}+y_{1}y_{2x}+y_{1x}y_{2},y_{2xx})\leq C\delta^{2}\|y_{2xx}\|+C\delta \|y_{2x}\|^{2}+C\delta \|y_{2}\|_{2}^{2}
\leq C\delta(1+ \|y_{2}\|_{2}^{2}).
\end{split}
\end{equation*}
Since
\begin{equation*}
\begin{split}
(y_{2}y_{2x},y_{2xx})=-\frac{1}{2}\int_{\mathbb{R}}y_{2x}^{3}dx\leq \varepsilon\|y_{2}\|_{3}^{2}+C\|y_{2}\|^{\frac{22}{5}}
\end{split}
\end{equation*}
By taking $0<\varepsilon<<1,$ we have
\begin{equation*}
\begin{split}
\partial_{t}\|y_{2}\|^{2}\leq C\delta(1+\|y_{2}\|^{2}),
\end{split}
\end{equation*}
this implies that if we $0<\delta<<1,$ we have
\begin{equation*}
\begin{split}
\sup\limits_{t\in [0,T]}\|y_{2}\|^{2}\leq 2(1+R^{2}).
\end{split}
\end{equation*}
By taking $0<\varepsilon<<1$ and $0<\delta<<1,$ we have
\begin{equation*}
\begin{split}
\partial_{t}(\|y_{2}\|^{2}+\|y_{2x}\|^{2})+a(\|y_{2}\|^{2}+\|y_{2x}\|^{2})+(\|y_{2xx}\|^{2}+\|y_{2xxx}\|^{2})\leq C(R,T),
\end{split}
\end{equation*}
this implies that
\begin{equation*}
\begin{split}
\sup\limits_{t\in [0,T]}\|y_{2}\|_{1}^{2}+\int_{0}^{T}\|y_{2}\|_{3}^{2}dt\leq C(R,T).
\end{split}
\end{equation*}
Moreover, on the event $\Gamma_{\delta},$ we have
\begin{equation*}
\begin{split}
\sup\limits_{t\in [0,T]}\|u\|_{1}^{2}+\int_{0}^{T}\|u\|_{3}^{2}dt\leq C(R,T).
\end{split}
\end{equation*}
\par
\textit{Step 2.} Now we prove that there exist sufficiently large $T>0$ and small $\delta>0$
such that $\|u(T)\|_{1}<d$ on the event $\Gamma_{\delta}$.

\par
Indeed, let $y_{3}$ be the solution of unforced hyperviscous Burgers equation
\begin{equation*}
\begin{array}{l}
\left\{
\begin{array}{lll}
\partial_{t}y_{3}+ay_{3}+y_{3xxxx}+y_{3}y_{3x}=0\\
y_{3}(x,0)=u_{0}
\end{array}
\right.
\end{array}
\begin{array}{lll}
\textrm{in}~\mathbb{R},
\\\textrm{in}~\mathbb{R},
\end{array}
\end{equation*}
by the similar argument as in Proposition \ref{pro4}, we have
\begin{equation}\label{16}
\begin{split}
\|y_{3}(t)\|\leq e^{-at}\|u_{0}\|.
\end{split}
\end{equation}
Since
\begin{equation*}
\begin{split}
(y_{3}y_{3x},y_{3xx})\leq \frac{1}{2}\|y_{3xxx}\|^{2}+C\|y_{3}\|^{2}+C\|y_{3}\|^{\frac{22}{5}},
\end{split}
\end{equation*}
we have
\begin{equation*}
\begin{split}
\partial_{t}\|y_{3x}\|^{2}+(2a\|y_{3x}\|^{2}+\|y_{3xxx}\|^{2})\leq C\|y_{3}\|^{2}+C\|y_{3}\|^{\frac{22}{5}}\leq C(u_{0})e^{-2at}.
\end{split}
\end{equation*}
Combining this with \eqref{16}, we can prove that
\begin{equation*}
\begin{split}
\|y_{3}(t)\|_{1}^{2}\leq C(u_{0})e^{-2at}(1+t),
\end{split}
\end{equation*}
we can choose a large $T=T(R,d)$ such that $\|y_{3}(T)\|_{1}\leq \frac{d}{4}.$
\par
Now we estimate the difference $y(t):=y_{2}(t)-y_{3}(t),$ then $y$ is the solution of the equation
\begin{equation}
\begin{array}{l}
\left\{
\begin{array}{lll}
\partial_{t}y+ay+y_{xxxx}=-ay_{1}-y_{1xxxx}-(uu_{x}-y_{3}y_{3x})\\
y(x,0)=0
\end{array}
\right.
\end{array}
\begin{array}{lll}
\textrm{in}~\mathbb{R},
\\\textrm{in}~\mathbb{R}.
\end{array}
\end{equation}
On the event $\Gamma_{\delta},$ we have
\begin{equation*}
\begin{split}
&-a(y_{1},y)\leq C\|y_{1}\|^{2}+\varepsilon \|y\|^{2}\leq C\delta^{2}+\varepsilon \|y\|^{2},\\
&-(y_{1xxxx},y)=-(y_{1xx},y_{xx})\leq C\|y_{1xx}\|^{2}+\varepsilon \|y\|_{2}^{2}\leq C\delta^{2}+\varepsilon \|y\|_{2}^{2},\\
&-(uu_{x}-y_{3}y_{3x},y)=-((y_{1}+y)u_{x}+y_{3}(y_{1}+y)_{x},y)\\
&~~~~~~\leq C\|y_{1}\|_{1}\|u\|_{1}\|y\|+C\|u\|_{2}\|y\|^{2}+C\|y_{1}\|_{1}^{2}+C(\|y_{3}\|_{1}^{2}+\|y_{3}\|_{2})\|y\|^{2}.
\end{split}
\end{equation*}
By taking $0<\varepsilon<<1$, these estimates imply that
\begin{equation*}
\begin{split}
\partial_{t}\|y\|^{2}\leq C(\delta^{2}+(\|u\|_{1}^{2}+\|u\|_{2}+\|y_{3}\|_{1}^{2}+\|y_{3}\|_{2})\|y\|^{2}),
\end{split}
\end{equation*}
then we have
\begin{equation*}
\begin{split}
\|y(T)\|^{2}\leq C\delta^{2}.
\end{split}
\end{equation*}
On the event $\Gamma_{\delta},$ we have
\begin{equation*}
\begin{split}
&-a(y_{1},-y_{xx})\leq C\|y_{1}\|^{2}+\varepsilon \|y_{xx}\|^{2}\leq C\delta^{2}+\varepsilon \|y_{xx}\|^{2},\\
&-(y_{1xxxx},-y_{xx})=-(y_{1xxx},y_{xxx})\leq C\|y_{1xxx}\|^{2}+\varepsilon \|y_{xxx}\|^{2}\leq C\delta^{2}+\varepsilon \|y_{xxx}\|^{2},\\
&-(uu_{x}-y_{3}y_{3x},-y_{xx})=-((y_{1}+y)u_{x}+y_{3}(y_{1}+y)_{x},y_{xx})\\
&-(y_{1}u_{x},y_{xx})\leq \delta\|u\|_{1}\|y_{xx}\|\leq C\delta^{2}\|u\|_{1}^{2}+\varepsilon \|y_{xx}\|^{2},\\
&-(u_{x}y,y_{xx})\leq \|u\|_{1}\|y\|_{1}\|y_{xx}\|\leq C_{\varepsilon}\|u\|_{1}^{2}\|y\|_{1}^{2}+\varepsilon \|y_{xx}\|^{2},\\
&-(y_{3}y_{1x},y_{xx})\leq \|y_{3}\|_{1}\|y_{1}\|_{1}\|y_{xx}\|\leq C_{\varepsilon}\delta^{2}\|y_{3}\|_{1}^{2}+\varepsilon \|y_{xx}\|^{2},\\
&-(y_{3}y_{x},y_{xx})\leq \|y_{3}\|_{1}\|y_{x}\|\|y_{xx}\|\leq C_{\varepsilon}\|y_{3}\|_{1}^{2}\|y_{x}\|^{2}+\varepsilon \|y_{xx}\|^{2}.
\end{split}
\end{equation*}
By taking $0<\varepsilon<<1$, these estimates imply that
\begin{equation*}
\begin{split}
\partial_{t}\|y\|_{1}^{2}\leq C(\delta^{2}+(\|u\|_{1}^{2}+\|u\|_{2}+\|y_{3}\|_{1}^{2}+\|y_{3}\|_{2})\|y\|_{1}^{2}),
\end{split}
\end{equation*}
then we have
\begin{equation*}
\begin{split}
\|y(T)\|_{1}^{2}\leq C\delta^{2}.
\end{split}
\end{equation*}
We choose $0<\delta<<\min\{1,\frac{d}{4\sqrt{C}}\}$, we have $\|y(T)\|_{1}\leq\frac{d}{4},$ and $\|u(T)\|_{1}\leq d$ on the event $\Gamma_{\delta}.$
\par
\textit{Step 3.} Since we have proven that
\begin{equation*}
\begin{split}
P_{T}(u_{0},B_{V}(0,d))\geq \mathbb{P}(\Gamma_{\delta})
\end{split}
\end{equation*}
for all $u_{0}\in B_{V}(0,R).$ Now, we prove $\mathbb{P}(\Gamma_{\delta})>0.$
Let us choose $N\geq1$ large enough so that
\begin{equation*}
\begin{split}
\|Q_{N}h\|_{3}\leq\sum\limits_{i=N+1}^{\infty}|(h,e_{i})|\|e_{i}\|_{3}<\frac{\delta}{3T}.
\end{split}
\end{equation*}
According to the independence of $Q_{N}W$ and $P_{N}W$, we have
\begin{equation*}
\begin{split}
\mathbb{P}(\Gamma_{\delta})
\geq \mathbb{P}(\sup\limits_{t\in [0,T]}\|P_{N}ht+P_{N}W(t)\|_{3}<\frac{\delta}{3})
\times\mathbb{P}(\sup\limits_{t\in [0,T]}\|Q_{N}W(t)\|_{3}<\frac{\delta}{3}).
\end{split}
\end{equation*}
Condition \eqref{41} implies that $\mathbb{P}(\sup\limits_{t\in [0,T]}\|P_{N}ht+P_{N}W(t)\|_{3}<\frac{\delta}{3})>0.$ It follows from $B_{3}<+\infty$ and the property of Gaussian measure that $\mathbb{P}(\sup\limits_{t\in [0,T]}\|Q_{N}W(t)\|_{3}<\frac{\delta}{3})>0.$ These show $\mathbb{P}(\Gamma_{\delta})>0.$
\par
This completes the proof.
\end{proof}
\par
\par
It follows from \cite[Proposition 3.3]{Shi08} that the above facts imply that the recurrence holds for the extension $(\mathbf{u}_{t},\mathbb{P}_{\mathbf{u}}).$ Now, we state that the extension $(\mathbf{u}_{t}, \mathbb{P}_{\mathbf{u}})$ of $(u_{t}, \mathbb{P}_{u})$
is irreducible.
\begin{proposition}\label{pro11}
For any $R,d>0,$ there exist constants $p,T>0$ and an integer $N\geq1$ such that
\begin{equation*}
\begin{split}
\mathbb{P}_{\mathbf{u}}(\mathbf{u}(T)\in B_{V}(0,d)\times B_{V}(0,d))\geq p
\end{split}
\end{equation*}
holds for any $\mathbf{u}\in B_{V}(0,R)\times B_{V}(0,R)$, provided that \eqref{41} holds.
\end{proposition}
\begin{proof}[Proof of Proposition \ref{pro11}]
Without loss of generality, we can assume $R\geq d.$
We define the events
\begin{equation*}
\begin{split}
&G_{d}(T):=\{\|\tilde{u}(T)\|_{1}\leq d\},\\
&G_{d}^{\prime}(T):=\{\|\tilde{u}^{\prime}(T)\|_{1}\leq d\},\\
&E_{\rho}:=\{\mathcal{E}_{\tilde{u}}^{\psi}(t)\leq M(\|u\|^{2}+\|u\|^{2p_{1}}+1)+Kt+\rho\}\cap \{\mathcal{E}_{\tilde{u}^{\prime}}^{\psi}(t)\leq M(\|u^{\prime}\|^{2}+\|u^{\prime}\|^{2p_{1}}+1)+Kt+\rho\},\\
&\mathcal{N}:=\{\tilde{v}(t)\neq \tilde{u}^{\prime}(t)~{\rm{for~some~}}t\geq0\}.
\end{split}
\end{equation*}
By the similar arguments as in \cite{Shi08}, \cite{M2014} and \cite{N2}, we can prove the following facts:
\par
(1) For any $u\in B_{V}(0,R)$ and $k\geq1.$
\begin{equation*}
\begin{split}
\mathbb{P}_{\mathbf{u}}(G_{ \frac{d}{2}}(kT))\wedge\mathbb{P}_{\mathbf{u}}(G^{\prime}_{ \frac{d}{2}}(kT))
\geq p_{0}.
\end{split}
\end{equation*}
\par
(2) For any $r>0,$ there exists $k\geq 1$ such that
\begin{equation*}
\begin{split}
G_{\frac{d}{2}}(kT)E_{r}\mathcal{N}^{c}\subset G_{d}(kT)G_{d}^{\prime}(kT)
\end{split}
\end{equation*}
for any $\mathbf{u}\in B_{V}(0,R)\times B_{V}(0,R).$
\par
With the help of the above facts and by the similar arguments as in \cite{Shi08}, \cite{M2014} and \cite{N2}, we can complete the proof.
\end{proof}

\begin{definition}
Let $(u(t),\mathbb{P}_{u})$ be a Markov process in a separable space $X$, and $F:X\rightarrow[1,+\infty)$
be a continuous functional such that
\begin{equation*}
\begin{split}
\lim\limits_{\|u\|_{X}\rightarrow+\infty}F(u)=+\infty.
\end{split}
\end{equation*}
Then, $F$ is said to be a \textbf{Lyapunov function} for $(u(t),\mathbb{P}_{u})$, if there are positive constants $t_{*},R_{*},C_{*},$ and $q_{*}<1$
such that
\begin{equation*}
\begin{split}
&\mathbb{E}_{u}F(u(t_{*}))\leq q_{*}F(u),~~\forall \|u\|\geq R_{*},\\
&\mathbb{E}_{u}F(u(t))\leq C_{*},~~\forall \|u\|\leq R_{*},t\geq0.
\end{split}
\end{equation*}
\end{definition}

\begin{proposition}\label{pro4}
The functional $F(u):=1+\|u\|_{1}^{2}+\|u\|^{2p_{2}}$ is a Lyapunov function for $(u(t),\mathbb{P}_{u})$
corresponding to \eqref{KSE}, where $p_{2}=\frac{11}{5}$.
\end{proposition}
\begin{proof}[Proof of Proposition \ref{pro4}]
By applying It\^{o} formula to $\|u_{x}\|^{2}$, it holds that
\begin{equation*}
\begin{split}
d\|u_{x}\|^{2}+2[a\|u_{x}\|^{2}+\|u_{xxx}\|^{2}]dt=2[(h_{x},u_{x})+(uu_{x},u_{xx})]dt+B_{2}dt+2(-u_{xx},dW),
\end{split}
\end{equation*}
where $B_{2}=\sum\limits_{i=1}^{\infty}b_{i}^{2}\|e_{ix}\|^{2}$.
Noting the fact
\begin{equation*}
\begin{split}
(uu_{x},u_{xx})\leq \frac{1}{2}\|u_{xxx}\|^{2}+C\|u\|^{2}+C\|u\|^{2p_{2}},
\end{split}
\end{equation*}
then, with the help of Proposition \ref{pro5}, we have
\begin{equation*}
\begin{split}
\frac{d}{dt}\mathbb{E}\|u_{x}(t)\|^{2}+a\mathbb{E}\|u_{x}\|^{2}+\mathbb{E}\|u_{xxx}\|^{2}
&\leq C\mathbb{E}\|u(t)\|^{2}+C\mathbb{E}\|u(t)\|^{2p_{2}}+C(\|h_{x}\|^{2}+B_{2})\\
&\leq C\mathbb{E}\|u(t)\|^{2p_{2}}+C(1+\|h\|_{1}^{2}+B_{2})\\
&\leq Ce^{-p_{2}at}\|u_{0}\|^{2p_{2}}+D,
\end{split}
\end{equation*}
where $D$ doesn't depend on $u_{0}$. This implies that
\begin{equation*}
\begin{split}
\mathbb{E}\|u_{x}(t)\|^{2}\leq e^{-at}\|u_{0x}\|^{2}+e^{-at}\|u_{0}\|^{2p_{2}}+\frac{D}{a},
\end{split}
\end{equation*}
namely, we have
\begin{equation*}
\begin{split}
\mathbb{E}\|u(t)\|_{1}^{2}\leq e^{-at}(\|u_{0}\|_{1}^{2}+\|u_{0}\|^{2p_{2}})+\frac{D}{a},
\end{split}
\end{equation*}
then we have
\begin{equation}\label{3}
\begin{split}
\mathbb{E}F(u(t))\leq e^{-at}F(u_{0})+E
\end{split}
\end{equation}
where $E$ doesn't depend on $u_{0}$. There exist constants $t_{\ast},R_{\ast}$ such that $e^{-at_{\ast}}\leq \frac{1}{4},$
$F(u_{0})\geq 4E$ for any $\|u_{0}\|_{1}\geq R_{\ast},$ namely, we have
\begin{equation*}
\begin{split}
\mathbb{E}F(u(t_{\ast}))\leq \frac{1}{2}F(u_{0}).
\end{split}
\end{equation*}
Moreover, according to \eqref{3}, there exists a $C_{\ast}$ such that for any $t\geq0,\|u_{0}\|_{1}\leq R_{\ast},$
\begin{equation*}
\begin{split}
\mathbb{E}F(u(t))\leq C_{\ast}.
\end{split}
\end{equation*}
\end{proof}
\par
It follows from Proposition \ref{pro11}, Proposition \ref{pro4} and Proposition
3.3 in \cite{Shi08} that the recurrence property holds for $(u(t), \mathbb{P}_{u})$.
\begin{proposition}
For any $d>0,$ there exist constants $C,\delta>0$ and an integer
$N\geq1$ depending on $d, a, h, B_{1}, B_{2}, B_{3}$ such that
\begin{equation*}
\begin{split}
\mathbb{E}_{\textbf{u}}e^{\delta \tau_{\textbf{B}}}\leq C(1+\|u\|_{1}^{2}+\|u\|^{2p_{2}}+\|u^{\prime}\|_{1}^{2}+\|u^{\prime}\|^{2p_{2}})
\end{split}
\end{equation*}
for any $\textbf{u}\in V\times V,$ provided that \eqref{41} holds.
\end{proposition}
\subsection{Polynomial squeezing}
We define the following stopping times
\begin{equation*}
\begin{split}
&\sigma:=\tilde{\tau}\wedge \sigma_{1},\\
&\tilde{\tau}:=\tau^{\tilde{u}}\wedge \tau^{\tilde{u}^{\prime}},\\
&\sigma_{1}:=\inf\{t\geq 0~|~\tilde{u}^{\prime}(t)\neq \tilde{v}(t)\}
\end{split}
\end{equation*}
and the following events
\begin{equation*}
\begin{split}
&\mathcal{Q}_{k}^{\prime}:=\{kT\leq \sigma\leq (k+2)T, \sigma_{1}\geq \tilde{\tau}\},\\
&\mathcal{Q}_{k}^{\prime\prime}:=\{kT\leq \sigma\leq (k+2)T, \sigma_{1}< \tilde{\tau}\}, k\geq0.
\end{split}
\end{equation*}
For $\mathcal{Q}_{k}^{\prime},\mathcal{Q}_{k}^{\prime\prime},$ we have the following property.
\begin{proposition}\label{pro12}
For any $q>1,$ there exist large enough constants $\rho,L,T>0$ and small enough $d>0$ such that
\begin{equation*}
\begin{split}
\mathbb{P}_{\mathbf{u}}(\mathcal{Q}_{k}^{\prime})\vee\mathbb{P}_{\mathbf{u}}(\mathcal{Q}_{k}^{\prime\prime})\leq \frac{1}{2(k+2)^{q}}
\end{split}
\end{equation*}
for any $\mathbf{u}\in \bar{B}_{H}(0,d)\times\bar{B}_{H}(0,d),$ provided that \eqref{41} holds.
\end{proposition}
\begin{proof}[Proof of Proposition \ref{pro12}]
The proof is divided into the following steps.
\par
\textbf{Step 1.} Estimate for $\mathbb{P}_{\mathbf{u}}(\mathcal{Q}_{k}^{\prime})$.
\par
According to Proposition \ref{proT1}, we have
\begin{equation*}
\begin{split}
\mathbb{P}_{\mathbf{u}}(\mathcal{Q}_{k}^{\prime})&\leq \mathbb{P}_{\mathbf{u}}(kT\leq \tilde{\tau}<+\infty)\\
&\leq \mathbb{P}_{u}(kT\leq \tau^{\tilde{u}}<+\infty)\\
&\leq CQ_{q}(\rho+kLT+1)\\
&\leq \frac{1}{2(k+2)^{q}}
,
\end{split}
\end{equation*}
where we take $\rho, L, T$ large enough.
\par
\textbf{Step 2.} Estimate for $\mathbb{P}_{\mathbf{u}}(\mathcal{Q}_{0}^{\prime\prime})$.
\par
For any $\mathbf{u}=(u,u^{\prime})\in H\times H,$ by the same argument as in Proposition \ref{pro10} and \eqref{13}, we have
\begin{equation}\label{23}
\begin{split}
\mathbb{P}_{\mathbf{u}}(\mathcal{Q}_{0}^{\prime\prime})
&\leq\mathbb{P}_{\mathbf{u}}(0\leq \sigma_{1}\leq T)\\
&\leq 2\mathbb{P}_{\mathbf{u}}(\tau^{u}<\infty)
+2\mathbb{P}_{\mathbf{u}}(\tau^{u^{\prime}}<\infty)
+\mathbb{P}_{\mathbf{u}}(\tau^{\hat{u}^{\prime}}<\infty)
+2\|\mathbb{P}-\Phi^{u,u^{\prime}}_{\ast}\mathbb{P}\|_{var}\\
&\leq CQ_{q}(\rho +1)+\frac{1}{2}\left(\exp(Ce^{C(\rho+\|u\|^{2p_{1}}+\|u^{\prime}\|^{2p_{1}}+1)}d^{2})-1 \right)^{\frac{1}{2}}.
\end{split}
\end{equation}
By taking $\rho$ large enough and $d$ small enough, we have
\begin{equation*}
\begin{split}
\mathbb{P}_{\mathbf{u}}(\mathcal{Q}_{0}^{\prime\prime})
\leq \frac{1}{2^{q+1}}.
\end{split}
\end{equation*}
\par
\textbf{Step 3.} Estimate for $\mathbb{P}_{\mathbf{u}}(\mathcal{Q}_{k}^{\prime\prime})$.
\par
Indeed, by the Markov property we have
\begin{equation*}
\begin{split}
\mathbb{P}_{\mathbf{u}}(\mathcal{Q}_{k}^{\prime\prime})=\mathbb{P}_{\mathbf{u}}(\mathcal{Q}_{k}^{\prime\prime},\sigma\geq kT)
=\mathbb{E}_{\mathbf{u}}(\mathbb{I}_{\mathcal{Q}_{k}^{\prime\prime}}\cdot\mathbb{I}_{\sigma\geq kT})
\leq\mathbb{E}_{\mathbf{u}}(\mathbb{I}_{\sigma\geq kT}\cdot \mathbb{P}_{\mathbf{u}(kT)}(0\leq \sigma_{1}\leq T))
,
\end{split}
\end{equation*}
combining this with \eqref{23}, we have
\begin{equation*}
\begin{split}
\mathbb{P}_{\mathbf{u}}(\mathcal{Q}_{k}^{\prime\prime})
\leq &2\mathbb{E}_{\mathbf{u}}(\mathbb{I}_{\sigma\geq kT}\cdot \mathbb{P}_{\mathbf{u}(kT)}(\tau^{u}<\infty))
+2\mathbb{E}_{\mathbf{u}}(\mathbb{I}_{\sigma\geq kT}\cdot \mathbb{P}_{\mathbf{u}(kT)}(\tau^{u^{\prime}}<\infty))\\
&+\mathbb{E}_{\mathbf{u}}(\mathbb{I}_{\sigma\geq kT}\cdot \mathbb{P}_{\mathbf{u}(kT)}(\tau^{\hat{u}^{\prime}}<\infty))
+2\mathbb{E}_{\mathbf{u}}(\mathbb{I}_{\sigma\geq kT}\cdot \|\mathbb{P}-\Phi^{\tilde{u}(kT),\tilde{u}^{\prime}(kT)}_{\ast}\mathbb{P}\|_{var})\\
:=&2I_{1}+2I_{2}+I_{3}+2I_{4}.
\end{split}
\end{equation*}
\par
For the term $I_{4},$ let $u_{k},u_{k}^{\prime}$ be the solutions of \eqref{KSE} issued from $\tilde{u}(kT),\tilde{u}^{\prime}(kT),$ respectively,
$v_{k}$ be the solution of \eqref{F1} issued from $\tilde{u}^{\prime}(kT),$
$\tau^{k}:=\tau^{u_{k}}\wedge\tau^{u_{k}^{\prime}}\wedge\tau^{v_{k}}.$ Define $w_{k}(t):=u_{k}(t)-v_{k}(t),$ when $0\leq t\leq \tau^{k},$ with the help of Theorem \ref{FP} with
\begin{equation}\label{19}
\varepsilon=\frac{aa_{0}}{4C_{\ast}(K+L)}(1\wedge \frac{1}{2M}),
\end{equation}
we have
\begin{equation*}
\begin{split}
\|w_{k}(t)\|_{1}^{2}\leq C\|w_{k}(0)\|_{1}^{2}\exp\left (-\frac{a}{2}t+C(\rho+\|\tilde{u}(kT)\|^{2p_{1}}+\|\tilde{u}^{\prime}(kT)\|^{2p_{1}}+1)\right)d^{2}.
\end{split}
\end{equation*}
Noting the fact on the set $\{\sigma\geq kT\},$ we have
\begin{equation*}
\begin{split}
\|\tilde{u}(kT)\|^{2p_{1}}&\leq \mathcal{E}_{\tilde{u}}^{\psi}(kT)\leq (K+L)kT+\rho+M(d^{2}+d^{2p_{1}}+1),\\
\|\tilde{u}^{\prime}(kT)\|^{2p_{1}}&\leq \mathcal{E}_{\tilde{u}^{\prime}}^{\psi}(kT)\leq (K+L)kT+\rho+M(d^{2}+d^{2p_{1}}+1),\\
\|w_{k}(0)\|_{1}^{2}&\leq Ce^{C(\rho+d^{2p_{1}}+1)}d^{2}e^{-\frac{a kT}{2}},
\end{split}
\end{equation*}
the above facts imply that for $t\leq \tau^{k}$ on the set $\{\sigma\geq kT\},$ we have
\begin{equation*}
\begin{split}
\|w_{k}(t)\|_{1}^{2}\leq Ce^{C(\rho+d^{2p_{1}}+1)}d^{2}e^{-\frac{a}{2}t-\frac{a kT}{2}}.
\end{split}
\end{equation*}
Since $\|A(t)\|\leq C\cdot 1_{t\leq \tau}\cdot\|w\|^{2}_{1}(\|u\|^{2}_{1}+\|v\|^{2}_{1})$ and $\mathcal{E}_{u_{k}}^{\psi}(t)+\mathcal{E}_{v_{k}}^{\psi}(t)\leq C(t+kT+\rho+d^{2p_{1}}+1)$ for $t\leq \tau^{k}$ on the set $\{\sigma\geq kT\},$
this leads to
\begin{equation*}
\begin{split}
\int_{0}^{\infty}\|A_{k}(t)\|^{2}dt
\leq C\exp\left (C(\rho+d^{2p_{1}}+1)\right) d^{2}e^{-\frac{a kT}{2}},
\end{split}
\end{equation*}
Hence,
\begin{equation*}
\begin{split}
I_{4}&\leq(\exp(Ce^{C(\rho+d^{2p_{1}}+1)}d^{2}e^{-\frac{a kT}{2}})-1)^{\frac{1}{2}}\\
&\leq(\exp(e^{-\frac{a kT}{2}})-1)^{\frac{1}{2}}({\rm{by~take~}}d{\rm{~small~enough}})\\
&\leq \sqrt{2}e^{-\frac{a kT}{4}}({\rm{by~take~}}T {\rm{~large~enough}})\\
&\leq \frac{1}{16(k+2)^{q}}.
\end{split}
\end{equation*}
\par
For the term $I_{1},$ by Proposition \ref{proT1} and the Markov property we have
\begin{equation*}
\begin{split}
\frac{1}{16(k+2)^{q}}&\geq\mathbb{P}_{\mathbf{u}}(kT\leq \tau^{u}<\infty)
=\mathbb{E}_{\mathbf{u}}[\mathbb{E}_{\mathbf{u}}(\mathbb{I}_{kT\leq \tau^{u}<\infty}|\mathcal{F}_{kT})]
=\mathbb{E}_{\mathbf{u}}[\mathbb{P}_{\mathbf{u}(kT)}(\tau^{u}<\infty)]\\
&\geq\mathbb{E}_{\mathbf{u}}[\mathbb{I}_{\sigma\geq kT}\cdot\mathbb{P}_{\mathbf{u}(kT)}(\tau^{u}<\infty)]
=I_{1}.
\end{split}
\end{equation*}
\par
For the term $I_{2},I_{3},$ we also have
\begin{equation*}
\begin{split}
I_{2},I_{3}\leq \frac{1}{16(k+2)^{q}}.
\end{split}
\end{equation*}
Combining the above estimates for $I_{i}$, we conclude
\begin{equation*}
\begin{split}
\mathbb{P}_{\mathbf{u}}(\mathcal{Q}_{k}^{\prime\prime})
\leq \frac{1}{2(k+2)^{q}}.
\end{split}
\end{equation*}
This completes the proof.

\end{proof}

\par
Now we can establish the polynomial squeezing in Theorem \ref{THACPM}.
\par
For any $p>1$, we take $q^{\prime}>p+1.$
Let us fix $N$ large enough so that Proposition \ref{pro12} and 2) in Theorem \ref{FP} hold with $\varepsilon$ as in \eqref{19}. Then, \eqref{PS4}
follows from the definition of $\sigma$ and 2) in Theorem \ref{FP}.
First, Proposition \ref{pro12} (with $q=q^{\prime}$) implies that
\begin{equation*}
\begin{split}
\mathbb{P}_{\mathbf{u}}(\sigma=\infty)
\geq 1-\sum_{k=0}^{\infty}\mathbb{P}_{\mathbf{u}}(\sigma\in [kT,(k+1)T])
\geq 1-2\sum_{k=0}^{\infty}\frac{1}{2(k+2)^{q^{\prime}}}:=\delta_{1}.
\end{split}
\end{equation*}
It follows from Proposition \ref{pro12} that
\begin{equation}\label{42}
\begin{split}
\mathbb{E}_{\mathbf{u}}(\mathbb{I}_{\sigma<\infty}\sigma^{p})
&\leq \sum_{k=0}^{\infty}\mathbb{E}_{\mathbf{u}}(\mathbb{I}_{\sigma\in [kT,(k+1)T]}\sigma^{p})\\
&\leq \sum_{k=0}^{\infty}\frac{1}{(k+2)^{q^{\prime}}}(k+1)^{p}T^{p}:=c.
\end{split}
\end{equation}
Noting for $\sigma<\infty,$ we have
\begin{equation*}
\begin{split}
\mathcal{E}_{\tilde{u}}(\sigma)\leq \mathcal{E}_{\tilde{u}}(0)+(K+L)\sigma+M(d^{2}+d^{2p_{1}}+1),\\
\mathcal{E}_{\tilde{u}^{\prime}}(\sigma)\leq \mathcal{E}_{\tilde{u}^{\prime}}(0)+(K+L)\sigma+M(d^{2}+d^{2p_{1}}+1),\\
\end{split}
\end{equation*}
thus, for $1\leq q\leq p,$ we have
\begin{equation*}
\begin{split}
\mathbb{E}_{\mathbf{u}}(\mathbb{I}_{\sigma<\infty}G(\mathbf{u}(\sigma))^{q})
\leq&C\mathbb{E}_{\mathbf{u}}(\mathbb{I}_{\sigma<\infty}(\|\tilde{u}(0)\|^{2}+\|\tilde{u}^{\prime}(0)\|^{2}+2(K+L)\sigma+2M(d^{2}+d^{2p_{1}}+1))^{q}\\
\leq&C\mathbb{E}_{\mathbf{u}}(\mathbb{I}_{\sigma<\infty}(1+\sigma)^{q})\\
\leq&C\mathbb{E}_{\mathbf{u}}(\mathbb{I}_{\sigma<\infty}\sigma^{p})\\
\leq &K,
\end{split}
\end{equation*}
where we have used \eqref{42}.

\section{Proof of Theorem \ref{THACPM}}
\begin{proof}[Proof of Theorem \ref{THACPM}]
The proof is divided into the following steps.
\par
\textit{Step 1.} Since $\{\rho=\infty\}=\{\sigma=\infty\},$ we have $\mathbb{P}_{\textbf{u}}(\rho=\infty)\geq \delta_{1}.$
For any $\textbf{u}\in \textbf{B},$
\begin{equation}\label{pm4}
\begin{split}
\mathbb{E}_{\textbf{u}}\mathbb{I}_{\{\rho<\infty\}}\rho^{p}\leq M,
\end{split}
\end{equation}
where $M$ is independent from $\textbf{u}.$
\par
Indeed, we have
\begin{equation*}
\begin{split}
\mathbb{E}_{\textbf{u}}\mathbb{I}_{\{\rho<\infty\}}\rho^{p}
=&\mathbb{E}_{\textbf{u}}\mathbb{I}_{\{\sigma<\infty\}}(\sigma+\tau_{\textbf{B}}\circ \theta_{\sigma})^{p}\\
\leq& 2^{p}\mathbb{E}_{\textbf{u}}\mathbb{I}_{\{\sigma<\infty\}}[\sigma^{p}+(\tau_{\textbf{B}}\circ \theta_{\sigma})^{p}]\\
=& 2^{p}\mathbb{E}_{\textbf{u}}\mathbb{I}_{\{\sigma<\infty\}}\sigma^{p}+2^{p}\mathbb{E}_{\textbf{u}}\mathbb{I}_{\{\sigma<\infty\}}(\tau_{\textbf{B}}\circ \theta_{\sigma})^{p}]
.
\end{split}
\end{equation*}
Moreover, according to Recurrence and the second inequality and third inequality in Polynomial squeezing, we can prove
\begin{equation*}
\begin{split}
\mathbb{E}_{\textbf{u}}\mathbb{I}_{\{\sigma<\infty\}}(\tau_{\textbf{B}}\circ \theta_{\sigma})^{p}
=&\mathbb{E}_{\textbf{u}}[\mathbb{E}_{\textbf{u}}\mathbb{I}_{\{\sigma<\infty\}}(\tau_{\textbf{B}}\circ \theta_{\sigma})^{p}|\mathcal{F}_{\sigma}]\\
=&\mathbb{E}_{\textbf{u}}\mathbb{I}_{\{\sigma<\infty\}}\mathbb{E}_{\textbf{u}}[(\tau_{\textbf{B}}\circ \theta_{\sigma})^{p}|\mathcal{F}_{\sigma}]\\
=&\mathbb{E}_{\textbf{u}}\mathbb{I}_{\{\sigma<\infty\}}\mathbb{E}_{\textbf{u}_{\sigma}}\tau_{\textbf{B}}^{p}~(\rm{strong~Markov~property})\\
\leq&\mathbb{E}_{\textbf{u}}\mathbb{I}_{\{\sigma<\infty\}}G(\textbf{u}_{\sigma})\leq K,
\end{split}
\end{equation*}
this implies \eqref{pm4}.
\par
\textit{Step 2.}  We define a sequence of stopping times $\rho_{k}=\rho_{k}(\textbf{u},\omega)$ as follows
\begin{equation*}
\begin{split}
\rho_{0}:=\tau_{\textbf{B}},~\rho_{k}:=\rho_{k-1}+\rho\circ\theta_{\rho_{k-1}},k\geq1.
\end{split}
\end{equation*}
It follows from the above fact
\begin{equation*}
\begin{split}
\rho_{k}=\rho_{0}+\sum\limits_{i=0}^{k-1}\rho\circ\theta_{\rho_{i}}
=\tau_{\textbf{B}}+\sum\limits_{i=0}^{k-1}\rho\circ\theta_{\rho_{i}}.
\end{split}
\end{equation*}
For any $\textbf{u}\in \textbf{X},$ \eqref{pm4} implies that
\begin{equation}\label{pm6}
\begin{split}
\mathbb{E}_{\textbf{u}}\mathbb{I}_{\{\rho_{k}<\infty\}}\rho_{k}^{p}
=&\mathbb{E}_{\textbf{u}}\mathbb{I}_{\{\rho_{k}<\infty\}}(\tau_{\textbf{B}}+\sum\limits_{i=0}^{k-1}\rho\circ\theta_{\rho_{i}})^{p}\\
\leq&(k+2)^{p}\mathbb{E}_{\textbf{u}}\mathbb{I}_{\{\rho_{k}<\infty\}}[\tau_{\textbf{B}}^{p}+\sum\limits_{i=0}^{k-1}(\rho\circ\theta_{\rho_{i}})^{p}]\\
\leq&(k+2)^{p}[\mathbb{E}_{\textbf{u}}\tau_{\textbf{B}}^{p}+\sum\limits_{i=0}^{k-1}\mathbb{E}_{\textbf{u}}(\rho\circ\theta_{\rho_{i}})^{p}\mathbb{I}_{\{\rho\circ\theta_{\rho_{i}}<\infty\}}]\\
\leq&(k+2)^{p}[G(\textbf{u})+kM]\\
\leq&(k+2)^{p+1}[G(\textbf{u})+M].
\end{split}
\end{equation}
\par
For any $\textbf{u}\in \textbf{X},$ define
\begin{equation*}
\begin{split}
\bar{k}=\bar{k}(\textbf{u},\omega)=\sup\{k\geq0:\rho_{k}(\textbf{u},\omega)<+\infty\},
\end{split}
\end{equation*}
then, since
\begin{equation}\label{pm1}
\begin{split}
\mathbb{P}_{\textbf{u}}(\rho_{k}<\infty)
\leq&(1-\delta_{1})\mathbb{P}_{\textbf{u}}(\rho_{k-1}<\infty)\\
\leq&(1-\delta_{1})^{k}\mathbb{P}_{\textbf{u}}(\rho_{0}<\infty)\\
\leq&(1-\delta_{1})^{k}:=a^{k},
\end{split}
\end{equation}
the Borel-Cantelli lemma implies
\begin{equation*}
\begin{split}
\bar{k}<+\infty~{\rm{for}}~\mathbb{P}_{\textbf{u}}-a.e.
\end{split}
\end{equation*}
\par
We define
\begin{equation*}
\ell=\ell(\textbf{u},\omega)=
\left\{
\begin{array}{ll}
\rho_{\bar{k}(\textbf{u},\omega)}(\textbf{u},\omega),~{\rm{if}}~\bar{k}(\textbf{u},\omega)<+\infty,\\
+\infty,~~~~~~~~~~~~{\rm{if}}~\bar{k}(\textbf{u},\omega)=+\infty.
\end{array}
\right.
\end{equation*}
If there exists $k$ such that $\rho_{k}(\textbf{u},\omega)<+\infty,\rho_{k+1}(\textbf{u},\omega)=+\infty,$ then for any $t\geq\rho_{k}(\textbf{u},\omega),$
\begin{equation*}
\begin{split}
\|u_{t}(\textbf{u},\omega)-u_{t}^{\prime}(\textbf{u},\omega)\|\leq C(t-\rho_{k}(\textbf{u},\omega)+1)^{-p}~{\rm{for}}~t\geq\rho_{k}(\textbf{u},\omega),
\end{split}
\end{equation*}
then this implies that
\begin{equation*}
\begin{split}
\|u_{t}(\textbf{u},\omega)-u_{t}^{\prime}(\textbf{u},\omega)\|\leq C(t-\ell+1)^{-p}~{\rm{for}}~t\geq\ell,
\end{split}
\end{equation*}
where $\textbf{u}\in \textbf{X}.$ For any $p_{0}<p$, \eqref{pm6} and \eqref{pm1} imply that
\begin{equation*}
\begin{split}
\mathbb{E}_{\textbf{u}}\ell^{p_{0}}
=&\sum_{k=0}^{\infty}\mathbb{E}_{\textbf{u}}\mathbb{I}_{\{\bar{k}=k\}}\rho_{k}^{p_{0}}\\
\leq&\sum_{k=0}^{\infty}\mathbb{E}_{\textbf{u}}\mathbb{I}_{\{\rho_{k}<\infty\}}\rho_{k}^{p_{0}}\\
\leq&\sum_{k=0}^{\infty}(\mathbb{E}_{\textbf{u}}\mathbb{I}_{\{\rho_{k}<\infty\}})^{\frac{1}{q^{\prime}}}(\mathbb{E}_{\textbf{u}}\mathbb{I}_{\{\rho_{k}<\infty\}}\rho_{k}^{p})^{\frac{1}{p^{\prime}}}\\
\leq&\sum_{k=0}^{\infty}\mathbb{P}_{\textbf{u}}(\rho_{k}<\infty)^{\frac{1}{q^{\prime}}}(\mathbb{E}_{\textbf{u}}\mathbb{I}_{\{\rho_{k}<\infty\}}\rho_{k}^{p})^{\frac{1}{p^{\prime}}}\\
\leq&\sum_{k=0}^{\infty}a^{\frac{k}{q^{\prime}}}(k+2)^{\frac{p+1}{p^{\prime}}}(G(\textbf{u})+M)^{\frac{1}{p^{\prime}}}\\
\leq&C(G(\textbf{u})+1).
\end{split}
\end{equation*}
\par
\textit{Step 3.}  For any $u,u^{\prime}\in X,$
\begin{equation*}
\begin{split}
\|P_{t}(u,\cdot)-P_{t}(u^{\prime},\cdot)\|_{L}^{*}\leq C(g(\|u\|)+g(\|u^{\prime}\|))(t+1)^{-p_{0}},~\forall t\geq0.
\end{split}
\end{equation*}
\par
Indeed, for any $f\in \mathcal{L}(X),\|f\|_{\mathcal{L}}\leq1,$
noting that
\begin{equation*}
\begin{split}
\mathbb{E}_{\textbf{u}}f(\Pi_{X}\textbf{u}_{t})=\mathcal{B}_{t}f(u),
\mathbb{E}_{\textbf{u}}f(\Pi_{X}^{\prime}\textbf{u}_{t})=\mathcal{B}_{t}f(u^{\prime}),
\end{split}
\end{equation*}
then
\begin{equation*}
\begin{split}
|(f,P_{t}(u,\cdot)-P_{t}(u^{\prime},\cdot))|=&|\mathbb{E}_{\textbf{u}}(f(u_{t})-f(u^{\prime}_{t}))|\\
\leq &\mathbb{E}_{\textbf{u}}|f(u_{t})-f(u^{\prime}_{t})|\\
\leq &2\mathbb{P}_{\textbf{u}}(\ell\geq \frac{t+1}{2})+\mathbb{E}_{\textbf{u}}\mathbb{I}_{\{\ell\leq \frac{t+1}{2}\}}|f(u_{t})-f(u^{\prime}_{t})|\\
\leq &C(g(\|u\|)+g(\|u^{\prime}\|))((t+1)^{-p}+(t+1)^{-p_{0}})\\
\leq &C(g(\|u\|)+g(\|u^{\prime}\|))(t+1)^{-p_{0}},~~\forall t\geq0
.
\end{split}
\end{equation*}
\par
\textit{Step 4.} For any $u,u^{\prime}\in X,$
\begin{equation}\label{pm3}
\begin{split}
\|P_{t}(u,\cdot)-P_{s}(u^{\prime},\cdot)\|_{L}^{*}\leq C(g(\|u\|)+g(\|u^{\prime}\|))(t+1)^{-p_{0}},~s\geq t\geq0.
\end{split}
\end{equation}
Indeed, for any $f\in \mathcal{L}(X),\|f\|_{\mathcal{L}}\leq1,$ we have $s\geq t\geq 0,$
\begin{equation*}
\begin{split}
|(f,P_{t}(u,\cdot)-P_{s}(u^{\prime},\cdot))|
&=|\int_{X}P_{s-t}(u^{\prime},dz)\int_{X}(P_{t}(u,dv)-P_{t}(z,dv))f(v)|\\
&\leq C(t+1)^{-p_{0}}\int_{X}P_{s-t}(u^{\prime},dz)[g(\|u\|)+g(\|z\|)]\\
&\leq C(t+1)^{-p_{0}}[g(\|u\|)+\mathbb{E}_{u^{\prime}}g(\|u_{s-t}\|)]\\
&\leq C(t+1)^{-p_{0}}[g(\|u\|)+g(\|u^{\prime}\|)].
\end{split}
\end{equation*}
\par
\textit{Step 5.} $\{u_{t}(u,\omega)\}_{t\geq0}$ is polynomial mixing.
\par
Indeed, it follows from Prokhorov theorem that $\mathcal{P}(X)$ is a complete metric space, then there exists $\mu\in \mathcal{P}(X)$
independent from $u$ and is a stationary measure, and $P_{t}(u,\cdot)\rightarrow \mu$ as $t\rightarrow+\infty.$ We take $u^{\prime}=0$ in \eqref{pm3},
let $s\rightarrow+\infty,$ we have
\begin{equation*}
\begin{split}
\|P_{t}(u,\cdot)-\mu\|_{L}^{*}\leq C(g(\|u\|)+g(0))(t+1)^{-p_{0}},~~t\geq0.
\end{split}
\end{equation*}
\par
This completes the proof.
\end{proof}

~~\\
~~\\
~~\\
~~
\noindent \footnotesize {\bf Acknowledgements.}
\par
Peng Gao would like to thank the referees and the editors for their careful comments and useful suggestions.
Peng Gao would like to thank Professor Kuksin S, Shirikyan A, Nersesyan V and Zhao M for fruitful discussions.
Peng Gao thanks Professor Kuksin S for his invitation to Universit\'{e} Paris Cit\'{e}, and for his scientific supervision on ergodicity for random dynamical systems. Peng Gao thanks the financial support of the China Scholarship Council (No. 202406620219).
This work is supported by National Natural Science Foundation of China (Grant No. 12371188).


{\small
\begin{thebibliography}{99}

\bibitem{BCK14} Bakhtin, Y., Cator E., and Khanin K. Space-time stationary solutions for the Burgers equation. Journal of the American Mathematical Society 27.1 (2014): 193-238.

\bibitem{BL19} Bakhtin, Y., Li LY. Thermodynamic limit for directed polymers and stationary solutions of the Burgers equation. Communications on Pure and Applied Mathematics 72.3 (2019): 536-619.

\bibitem{DGR21} Dunlap, A., Graham, C., Ryzhik, L. Stationary solutions to the stochastic Burgers equation on the line. Communications in Mathematical Physics, 382(2), 875-949 (2021).

\bibitem{BFZ23} Brze\'{z}niak, Z., Ferrario, B., Zanella, M. Ergodic results for the stochastic nonlinear Schr\"{o}dinger equation with large damping. Journal of Evolution Equations, 23(1), 19 (2023).

\bibitem{B2014-hyp} Banerjee D, Ray S S. Transition from dissipative to conservative dynamics in equations of hydrodynamics[J]. Physical Review E, 2014, 90(4): 041001.

\bibitem{C1995-hyp} Chekhlov A, Yakhot V. Kolmogorov turbulence in a random-force-driven Burgers equation[J]. Physical Review E, 1995, 51(4): R2739.

\bibitem{Z1997-hyp} Zikanov O, Thess A, Grauer R. Statistics of turbulence in a generalized random-force-driven Burgers equation[J]. Physics of Fluids, 1997, 9(5): 1362-1367.

\bibitem{F2013-hyp} Frisch U, Ray S S, Sahoo G, et al. Real-space manifestations of bottlenecks in turbulence spectra[J]. Physical review letters, 2013, 110(6): 064501.

\bibitem{G2002-hyp} Gugg C, Kielh\"{o}fer H, Niggemann M. On the approximation of the stochastic Burgers equation[J]. Communications in mathematical physics, 2002, 230(1): 181-199.

\bibitem{B2019-hyp} Banerjee D. Fractional hyperviscosity induced growth of bottlenecks in energy spectrum of Burgers equation solutions[J]. The European Physical Journal B, 2019, 92(9): 209.

\bibitem{A1989} Abergel, F. Attractor for a Navier-Stokes flow in an unbounded domain. ESAIM: Mathematical Modelling and Numerical Analysis 23.3 (1989): 359-370.

\bibitem{Fla1} Albeverio S, Flandoli F, Sinai Y G, et al. An introduction to 3D stochastic fluid dynamics[J]. SPDE in Hydrodynamic: Recent Progress and Prospects: Lectures given at the CIME Summer School held in Cetraro, Italy August 29-September 3, 2005, 2008: 51-150.

\bibitem{M2014} Martirosyan D. Exponential mixing for the white-forced damped nonlinear wave equation[J]. Evolution Equations and Control Theory, 2014, 3(4): 645-670.

\bibitem{DaZ1} Da Prato G., Zabczyk J. Stochastic Equations in Infinite Dimensions, Cambridge University Press, 2014.

\bibitem{DaZ2} Da Prato G., Zabczyk J. Ergodicity for infinite dimensional systems. Vol. 229. Cambridge university press, 1996.

\bibitem{Deb1} Debussche A. Ergodicity results for the stochastic Navier-Stokes equations: an introduction[J]. Topics in Mathematical Fluid Mechanics: Cetraro, Italy 2010, Editors: Hugo Beirao da Veiga, Franco Flandoli, 2013: 23-108.

\bibitem{GPW} Gao, P. Polynomial mixing for the white-forced wave equation on the whole line. arXiv preprint arXiv:2412.13230 (2024).

\bibitem{HM06} Hairer M, Mattingly J C. Ergodicity of the 2D Navier-Stokes equations with degenerate stochastic forcing[J]. Annals of Mathematics, 2006: 993-1032.

\bibitem{HM08} Hairer M, Mattingly J C. Spectral gaps in Wasserstein distances and the 2D stochastic Navier-Stokes equations[J]. Annals of Probability, 2008, 36(6): 2050-2091.

\bibitem{HM11-1} Hairer M, Mattingly J C. A theory of hypoellipticity and unique ergodicity for semilinear stochastic PDEs[J]. Electronic Journal of Probability, 2011, 16: 23.

\bibitem{HM11-2} Hairer M, Mattingly J C, Scheutzow M. Asymptotic coupling and a general form of Harris' theorem with applications to stochastic delay equations[J]. Probability theory and related fields, 2011, 149: 223-259.

\bibitem{F1}Ferrario B. Invariant measures for a stochastic Kuramoto-Sivashinsky equation[J]. Stochastic analysis and applications, 2008, 26(2): 379-407.

\bibitem{FG15} F\"{o}ldes J, Glatt-Holtz N, Richards G, et al. Ergodic and mixing properties of the Boussinesq equations with a degenerate random forcing[J]. Journal of Functional Analysis, 2015, 269(8): 2427-2504.


\bibitem{KS12} Kuksin S, Shirikyan A. Mathematics of two-dimensional turbulence[M]. Cambridge University Press, 2012.

\bibitem{KNS1} Kuksin S, Nersesyan V, Shirikyan A. Exponential mixing for a class of dissipative PDEs with bounded degenerate noise[J]. Geometric and Functional Analysis, 2020, 30(1): 126-187.

\bibitem{Shi08} Shirikyan A. Exponential mixing for randomly forced partial differential equations: method of coupling[M]//Instability in Models Connected with Fluid Flows II. New York, NY: Springer New York, 2008: 155-188.

\bibitem{M1} Maslowski B, Seidler J. Invariant measures for nonlinear SPDE's: uniqueness and stability[J]. Archivum Mathematicum, 1998, 34(1): 153-172.

\bibitem{N2} Nersesyan V, Zhao M. Exponential mixing for the white-forced complex Ginzburg--Landau equation in the whole space[J]. SIAM Journal on Mathematical Analysis, 2024, 56(3): 3646-3678.

\bibitem{N3} Nersesyan V, Zhao M. Polynomial mixing for the white-forced Navier-Stokes system in the whole space. arXiv preprint arXiv:2410.15727 (2024).

\bibitem{PMJEE2005} Debussche A, Odasso C. Ergodicity for a weakly damped stochastic non-linear Schr\"{o}dinger equation[J]. Journal of Evolution Equations, 2005, 5: 317-356.

\bibitem{PMJEE2008} Nersesyan V. Polynomial mixing for the complex Ginzburg-Landau equation perturbed by a random force at random times[J]. Journal of Evolution Equations, 2008, 8: 1-29.

\bibitem{PMAMO} Nguyen H D. Polynomial mixing of a stochastic wave equation with dissipative damping[J]. Applied Mathematics \& Optimization, 2024, 89(1): 21.


\bibitem{PM2024} Nguyen H D. The inviscid limit for long time statistics of the one-dimensional stochastic Ginzburg-Landau equation[J]. arXiv preprint arXiv:2403.08951, 2024.

\bibitem{FPE} Foias C, Prodi G. Sur le comportement global des solutions non-stationnaires des \'{e}quations de Navier-Stokes en dimension $2$[J]. Rendiconti del Seminario matematico della Universita di Padova, 1967, 39: 1-34.



\end{thebibliography}}
\end{document}